\documentclass[11pt]{amsart}
\usepackage{amssymb}
\usepackage{amscd}
\usepackage{color}
\usepackage{comment}
\usepackage[dvipsnames]{xcolor}
\usepackage{url}
\usepackage{hyperref}
\usepackage{enumitem}

\usepackage{graphicx}
\usepackage{caption}
\usepackage{subcaption}

\newtheorem{theorem}{Theorem}[section]
\newtheorem{lemma}[theorem]{Lemma}
\newtheorem{proposition}[theorem]{Proposition}

\theoremstyle{definition}
\newtheorem{definition}[theorem]{Definition}

\newenvironment{demode}
{\noindent {{\it Proof of }}}%
{\par \hfill \fbox{}}

\theoremstyle{theorem}
\newtheorem*{theorem*}{Main Theorem}
\newtheorem*{question}{Question}

\theoremstyle{definition}

\theoremstyle{remark}
\newtheorem{remark}[theorem]{Remark}

\numberwithin{equation}{section}

\newcommand{\norm}[1]{\left\lvert\left\lvert#1\right\rvert\right\rvert}

\DeclareMathOperator*{\Span}{span \;}

\newcommand{\ran}{\textnormal{ran}}

\newcommand{\pe}[1]{\langle#1\rangle}
\newcommand{\EL}{\mathcal{L}}

\newcommand{\C}{\mathbb{C}}

\newcommand{\N}{\mathbb{N}}
\newcommand{\D}{\mathbb{D}}
\newcommand{\T}{\mathbb{T}}
\newcommand{\PR}{\textnormal{Re}}

\newcommand{\al}{\alpha}

\newcommand{\sumn}{\sum_{n=1}^{\infty}}

\newcommand{\summ}{\sum_{m=1}^{\infty}}
\newcommand{\suml}{\sum_{\ell=1}^{\infty}}
\newcommand{\sumk}{\sum_{k=1}^{\infty}}

\newcommand{\y}{\textnormal{\textbf{y}}}

\newcommand{\R}{\mathbb{R}}

\newcommand{\x}{\textnormal{\textbf{x}} }
\newcommand{\sumi}{\sum_{i=1}^\infty}
\newcommand{\sumj}{\sum_{j=1}^\infty}
\newcommand {\conm}[1]{{\{#1\}'}}
\newcommand{\biconm}[1]{{\{#1\}''}}

\newcommand{\X}{\mathcal{X}}

\newcommand{\inte}{\text{int}}

\begin{document}
	\title[Trace class perturbations of normal operators]{Hyperinvariant subspaces for trace class perturbations \newline  of normal operators and decomposability}

	\author{Eva A. Gallardo-Guti\'{e}rrez}
	\address{Eva A. Gallardo-Guti\'errez \newline
		Departamento de An\'alisis Matem\'atico y Matem\'atica Aplicada,\newline
		Facultad de Matem\'aticas,
		\newline Universidad Complutense de
		Madrid, \newline
		Plaza de Ciencias N$^{\underbar{\Tiny o}}$ 3, 28040 Madrid,  Spain
		\newline
		and Instituto de Ciencias Matem\'aticas ICMAT (CSIC-UAM-UC3M-UCM),
		\newline Madrid,  Spain } \email{eva.gallardo@mat.ucm.es}

	\author{F. Javier Gonz\'alez-Doña}
		\address{F. Javier González-Doña \newline
			Departamento de Matemáticas, \newline
			Escuela Politécnica Superior, \newline
			 Universidad Carlos III de Madrid, \newline
			   Avenida de la Universidad 30, 28911 Leganés, Madrid, Spain
		   	\newline
		   and Instituto de Ciencias Matem\'aticas ICMAT (CSIC-UAM-UC3M-UCM),
		   \newline Madrid,  Spain }
	\email{fragonza@math.uc3m.es}

	\thanks{Both authors are partially supported by Plan Nacional  I+D grant no. PID2022-137294NB-I00, Spain,
		the Spanish Ministry of Science and Innovation, through the ``Severo Ochoa Programme for Centres of Excellence in R\&D'' (CEX2019-000904-S \& CEX2023-001347-S) and from the Spanish National Research Council, through the ``Ayuda extraordinaria a Centros de Excelencia Severo Ochoa'' (20205CEX001).}
	

\date{September 2024, revised version January 2025}

\begin{abstract}
We prove that a large class of trace-class perturbations of diagonalizable normal operators  on a separable, infinite dimensional complex Hilbert space have non-trivial closed hyperinvariant subspaces. Moreover, a large subclass consists of decomposable operators in the sense of  Colojoar\u{a} and Foia\c{s} \cite{CF68}.
\end{abstract}

\maketitle

\section{Introduction}

A long-standing open problem is whether every rank-one perturbation of a diagonalizable normal operator on a separable infinite-dimensional complex Hilbert space has a non-trivial closed invariant subspace (see, for example, Problem 8K in the 1978 Pearcy's monograph \cite{Pearcy}). About fifteen years ago, remarkable results in this setting were proved by Foia\c{s}, Jung, Ko and Pearcy in \cite{FJKP07, FJKP08, FJKP11}, and subsequently by Fang and J. Xia for the case of finite-rank perturbations (see \cite{FX12}).  Recently, the authors in \cite{GG, GG2, GG3} have undertaken a thorough study based on characterising spectral subspaces which has allowed them to prove that a large class of finite-rank perturbations of diagonalizable normal operators are \emph{decomposable operators} in the sense of  Colojoar\u{a} and Foia\c{s} \cite{CF68}. As a consequence, every operator $T$ in such a class has a rich spectral structure and plenty of non-trivial closed hyperinvariant subspaces, which extends all the previously known results on such  question claimed as \emph{a stubbornly intractable problem} in \cite{FJKP07}.

\medskip

Nevertheless, the character of the finiteness rank of the compact operator plays a fundamental role in the aforementioned study. Indeed, studying the existence of invariant subspaces or the decomposability of perturbations of normal operators, even of hermitian operators, by \emph{general} compact operators is an old problem.
Livsi\u{c} solved the existence problem for nuclear perturbations of self-adjoint operators, Sahnovi\u{c} for Hilbert-Schmidt perturbations, and Gohberg and Krein, Macaev, and Schwartz for perturbations with compacts belonging to the Schatten von Neumann class $\mathcal{C}_p$, $1\leq p<\infty$ (see \cite{Dunford and Schwarz} for more on the subject).
In the setting of compact perturbations of normal operators $N$, if the spectrum of $N$ lies on a $C^2$ Jordan curve $\gamma$,  Radjabalipour and H. Radjavi \cite{Radjabalipour-Radjavi} showed that the linear bounded  operator $T=N+K$ where  $K$ is a compact operator belonging to $\mathcal{C}_p$, $1\leq p<\infty$, is  decomposable if and only if the spectrum $\sigma(T)$ does not fill the interior of $\gamma$  (see also \cite{Radjabalipour}).

\smallskip

But, in 1977 Herrero \cite{Herrero} proved that not every compact perturbation of a unitary operator is decomposable. So not every compact perturbation of a normal operator with a spectrum lying on a $C^2$ Jordan curve is decomposable, and the situation becomes even more hopeless if no assumption on the spectrum is required.

\medskip

In this framework, the aim of this work is studying trace-class perturbations of diagonalizable normal operators $(\mathcal{N}+\mathcal{C}_1)$ acting on a separable, infinite dimensional complex Hilbert space $H$ from the standpoint of view of the existence of invariant subspaces. Indeed, we will prove the that there exists a large subclass $(\mathcal{N}+\mathcal{C}_1)_{hs}$ having non-trivial closed \emph{hyperinvariant subspaces}, namely, each $T\in (\mathcal{N}+\mathcal{C}_1)_{hs}$ has non-trivial closed invariant subspaces which are also invariant for every operator in its commutant $\{T\}'$.  Moreover, for those $T\in (\mathcal{N}+\mathcal{C}_1)_{hs}$ such that the point spectrum of both $T$ and its adjoint $T^*$ is at most countable, we will show that they do have more structure regarding invariant subspaces, namely, they are decomposable operators. Recall that a bounded linear operator $T$ on $H$ is decomposable if splitting the spectrum $\sigma(T)$ provides a decomposition of $H$ in terms of
invariant subspaces. More precisely, if for every open cover $\{U,V\}$ of $\mathbb{C}$ there exist two closed invariant subspaces $H_1,H_2 \subseteq
H$ such that
\[
	\sigma(T\vert_{H_1}) \subseteq U \, \text{ and } \, \sigma(T\vert_{H_2}) \subseteq V,
\]
and $H= H_1+H_2$. Here $\sigma(T\vert_{H_i})$ denotes the spectrum of the restriction of $T$ to the invariant subspace $H_i$, $i=1,2$. It is worthy to point out that the sum decomposition is, in general, not direct, nor are the spectra of the restrictions necessarily disjoint. Decomposable operators were introduced by Foia\c{s} \cite{FOIAS} in the sixties as a generalization of spectral operators in the sense of Dunford \cite{Dunford and Schwarz} and  Foia\c{s}' original definition was somewhat more technical,
but equivalent to the one set down here (see \cite{LN00} for more on the subject).

\medskip

In order to state our main result, let $H$ be a separable, infinite dimensional complex Hilbert space and $\mathcal{E}= \{e_n\}_{n\geq 1}$  an (ordered) orthonormal basis of $H$ fixed. 

If $\Lambda = \{\lambda_n\}_{n\geq 1}$ is any bounded sequence in the complex plane $\C$, let $D_\Lambda$ denote  the diagonal operator with respect to $\mathcal{E}$ associated to $\Lambda$, namely,
$$D_\Lambda e_n = \lambda_n e_n,\qquad (n\geq 1).$$
For $k\geq 1$, let  $u_k$ and $v_k$  non zero vectors in $H$  and let us denote their Fourier coefficients with respect to $\mathcal{E}$ as
$$ u_k = \sumn \al_n^{(k)}e_n, \quad v_k = \sumn \beta_n^{(k)}e_n \qquad (k\geq 1).$$
Our main theorem in this work reads as follows:

\begin{theorem*}
Let $H$ be a separable, infinite dimensional complex Hilbert space, $\Lambda = (\lambda_n)_{n\geq 1}\subset \mathbb{C}$ a bounded sequence and $\{u_k\}_{k\geq 1}$, $\{v_k\}_{k\geq 1}$  non zero vectors in $H$. Assume
			 \begin{equation}\label{condicion}
			 \sum_{(n,k) \in \mathcal{N}_u} |\al_n^{(k)}|^2\log \left(1+\frac{1}{|\al_n^{(k)}|}\right) + \sum_{(n,k) \in \mathcal{N}_v}|\beta_n^{(k)}|^2\log\left(1+ \frac{1}{|\beta_n^{(k)}|}\right) < \infty,
			 \end{equation}
			 where $\mathcal{N}_u := \{(n,k) \in \N\times \N : \al_n^{(k)} \neq 0 \}$ and $\mathcal{N}_v := \{(n,k) \in \N\times \N : \beta_n^{(k)} \neq 0 \}$.
Then, the trace-class perturbation of $D_\Lambda$, $T = D_\Lambda + \sumk u_k\otimes v_k$, acting on $H$ by
\begin{equation}\label{forma}
T x= \left (D_\Lambda + \sumk u_k\otimes v_k \right ) x= D_\Lambda x +  \sumk \pe{x,v_k} u_k, \qquad (x\in H),
\end{equation}
has non trivial closed hyperinvariant subspaces provided that it is not a scalar multiple of the identity operator. Moreover, if both point spectrum $\sigma_p(T)$ and $\sigma_p(T^*)$ are at most countable, $T$ is decomposable.
\end{theorem*}

Before proceeding further, a couple of comments are in order. Firstly,   condition \eqref{condicion} implies that
\begin{equation}\label{consecuencia sumabilidad 1}
\sumn \sumk \left(|\al_n^{(k)}|^2 + |\beta_n^{(k)}|^2 \right) < \infty
\end{equation}
and
	\begin{equation}\label{consecuencia sumabilidad 2}
		\sum_{(n,k) \in \mathcal{N}_u} |\al_n^{(k)}|^2\log \left(\frac{1}{|\al_n^{(k)}|}\right) + \sum_{(n,k) \in \mathcal{N}_v}|\beta_n^{(k)}|^2\log\left( \frac{1}{|\beta_n^{(k)}|}\right) < \infty.
	\end{equation}
In particular, \eqref{consecuencia sumabilidad 1} yields that $\sumk (\norm{u_k}^2 + \norm{v_k}^2) < \infty,$ so the compact operator $K = \sumk u_k\otimes v_k $ is  trace-class. Secondly, trace-class perturbations of normal operators whose eigenvectors span $H$  are unitarily equivalent to those expressed by \eqref{forma}.

\medskip

\begin{remark}
It should be noted that the condition \eqref{condicion} is slightly different to the one imposed in \cite{GG2, GG3}.
This modification is necessary to ensure the validity of both \eqref{consecuencia sumabilidad 1} and \eqref{consecuencia sumabilidad 2}. An equivalent condition for obtaining these summability properties is to assume \eqref{consecuencia sumabilidad 2} and furthermore to suppose that there exist finite subsets $\mathfrak{N}_u, \mathfrak{N}_v \subset \N\times \N$ such that
$$ \sup\limits_{(n,k)\in (\N\times\N)\setminus \mathfrak{N}_u} |\al_n^{(k)}| < 1, \qquad  \sup\limits_{(n,k)\in (\N\times\N)\setminus \mathfrak{N}_v} |\beta_n^{(k)}| < 1.$$
\end{remark}

\medskip

As it was aforementioned, the main theorem of this work extends the results in \cite{FJKP07, FJKP08, FJKP11, FX12, GG2, GG3}. In particular, in the rank-one perturbation case, if $T = D_\Lambda + u_1\otimes v_1$ is non-scalar,  the summability condition \eqref{condicion} along with authors' contributions yields, in particular, that $T$ has non trivial closed hyperinvariant subspaces as far as
$$\sumn (|\al_n^{(1)}|^p + |\beta_n^{(1)}|^q )<\infty$$
for every $(p,q)\in (0,2]\times (0,2]\setminus \{(2, r), (r, 2):\; r\in(1,2]\}$ being, indeed, decomposable whenever the point spectrum of $T$ and $T^*$ are, at most, denumerable (see Figure \ref{fig:1}).
Moreover, in the infinite-rank perturbation case, it comprises the results by Klaja \cite{Klaja} and Theorem 1.2 of Albretch and Chevreau \cite{AA}.

\begin{figure}[htb]
\centering
  \includegraphics[width=.35\linewidth]{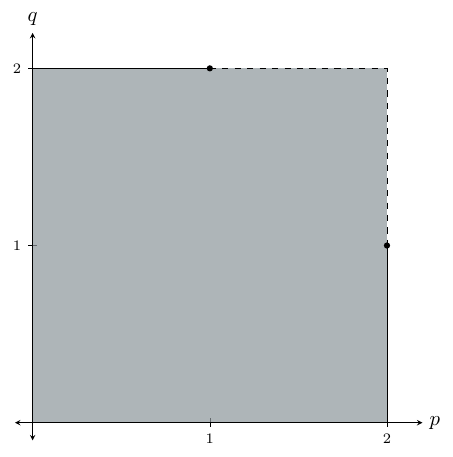}
  \caption{Decomposability in the rank-one perturbation case}
  \label{fig:1}
\vspace*{-0,2cm}
\end{figure}

Before closing this introductory section, we outline the strategy in order to prove our main theorem, pointing out the main differences with the previously aforementioned results.
As in the finite-rank case, the proof relies on the construction of a family of spectral idempotents associated to the operator $T=D_\Lambda+\sumk u_k\otimes v_k$ that is, idempotents $J$ belonging to the bicommutant of $T$ such that their ranges are spectral subspaces of $T$. These ranges will be the non-trivial closed hyperinvariant subspaces for $T$. Moreover, we will be able to construct a sufficiently rich Boolean algebra of such spectral idempotents, which will provide the decomposability as in  \cite{FJKP11}.

\smallskip

Accordingly, our main efforts are addressed to the construction of such operators and for such a task, we will make use of the so-called \textit{unconventional Dunford functional calculus}, but with substantial modifications.

\smallskip

The unconventional functional calculus was explicitly stated in \cite{FX12} and used in \cite{Klaja, AA} to provide non-trivial closed hyperinvariant subspaces and decomposability for certain compact perturbations of diagonalizable normal operators, as we already pointed out. The main idea consists of considering piecewise-differentiable closed curves $\gamma \subset \C$ intersecting the spectrum $\sigma(T)$ and defining a right-inverse $R(z)$ of $(T-zI)$ for every $z \in \gamma.$ This inverse turns out to be of the form
$$R(z) = (D_\Lambda-z I)^{-1} + K(z),$$
where $K(z)$ is a compact operator for each $z \in \gamma.$ In general, $(D_\Lambda-zI)^{-1}$ may be unbounded, so $R(z)$ would be unbounded as well. Nevertheless, the operator $R(z)$ can be integrated along the curve $\gamma$, due to the Borel functional calculus for normal operators (see \cite[Proposition 2.6]{Klaja}). Hence, the spectral idempotents (or other kind of operators with equivalent properties) are obtained as the operators defined via
$$\frac{1}{2\pi i} \int_\gamma R(z)dz.$$

\smallskip

Our approach will also rely on integrating expressions of the kind $(D_\Lambda-zI)^{-1} + K(z)$ throughout a curve (see \eqref{definicion idempotentes}), but under the assumptions of summability \eqref{condicion} the operator $K(z)$ may  not even be well defined for those $z \in \gamma\cap \sigma(T)$. Thus, we will consider a \textit{formal} right inverse of $(T-zI)$ and integrate it along a curve $\gamma$ which will allow us to obtain bounded operators $J$ (see Theorem \ref{idempotentes acotados}) and which will turn out to be spectral idempotents (Theorem \ref{teorema idempotentes}). The proof of this latter fact will be carried out through Section \ref{seccion 4}, for which a characterization of the spectral subspaces associated to $T$ will be needed (Theorem \ref{caracterizacion subespacios espectrales}).

\smallskip

Integrating formal right inverses of $(T-zI)$ already appeared implicitly in the authors' work \cite{GG3} (tracing back to \cite[Theorem 6.7]{GG}) in connection with the existence of spectral idempotents for finite rank perturbations of diagonalizable normal operators.  Nevertheless, in such a context, the definition of the spectral idempotents relies heavily on finite-dimensional linear algebra since it deals with the determinant and cofactors of a certain matrix $M_T$ associated to the operator (see \cite[Lemma 4.5 and Definition 4.8]{GG3}). Moreover, the expression of the determinant via the associated cofactors plays a decisive role in the construction.

\smallskip

For trace class perturbations of normal operators, the finite-dimensional linear algebra does not apply anymore and the tools coming from Fredholm theory of determinants seem not to be suitable to face the problem. Accordingly, the construction of the idempotents in this setting will depend on the invertibility of an auxiliary operator $I+Y(z)X(z)$, which plays a role similar to that of the matrix $M_T$ along the curve $\gamma$. In addition, the properties of continuity of the map $z\in \gamma \mapsto (I+Y(z)X(z))^{-1}$ will allow us to obtain bounds of conditionally-convergent series and an inversion formula (Theorem \ref{teorema formulas}) that will be fundamental to obtain the desired properties for the operators $J$. The construction of the operator $I+Y(z)X(z)$ and the study of its properties will be carried out in Section \ref{seccion 3}.

\section{Preliminaries}

In this section, we recall some preliminaries and prove a characterization of the spectral subspaces associated to closed sets of the complex plane for trace-class perturbations of diagonalizable normal along the lines of \cite[Theorem 2.1]{GG}.

Throughout this work, $H$  denotes an infinite dimensional separable complex  Hilbert space, $\EL(H)$ the Banach algebra of all bounded linear operators on $H$ and $\mathcal{E}= \{e_n\}_{n\geq 1}$ an (ordered) orthonormal basis of $H$ fixed.

Recall that compact operator $K \in \EL(H)$ belongs to the Schatten class $\mathcal{C}_p$ with $1\leq p < \infty$ if
$$\sumn a_n(K)^p<\infty,$$ where $a_n(K)$ denotes the singular values of $K$, namely, eigenvalues of the hermitian operator $|K|:=\sqrt{K^*K}$
Indeed, for every $1\leq p < \infty$, the norm
$$ \norm{K}_{\mathcal{C}_p} := \left( \sumn a_n(K)^p\right)^{1/p} < \infty$$
makes each Schatten class $\mathcal{C}_p$ a closed ideal of operators of $\EL(H)$.

For $p=1$, $\norm{ \; }_{\mathcal{C}_1}$ is the trace-class norm  and $K$ is trace-class if there exist two orthogonal sequences $\{x_i\}$ and $\{y_i\}$ in $H$  and positive real numbers $\left(\lambda _{i}\right)$ in $\ell ^{1}$ such that
$$
 x\mapsto T(x)=\sum _{i=1}^{\infty }\lambda _{i}\left\langle x,x_{i}\right\rangle y_{i},\quad \forall x\in H,
$$
where $(\lambda _{i})$ are the singular values of $K$, with each value repeated as often as its multiplicity.

\smallskip


As we mentioned in the introduction,  the proof of the Main Theorem will be carried out by constructing a sufficiently rich family of spectral idempotents of the operator $T$, that is, idempotents  with \emph{spectral subspaces} ranges. Spectral subspaces play an important role in order to produce non-trivial closed hyperinvariant subspaces and to introduce them, let us recall a few relevant notions from Local Spectral Theory (we refer to the monograph \cite{LN00} for more on the subject).

\smallskip

A linear bounded operator $T$ in $H$ has the \emph{single-valued extension property} (SVEP) if for every connected open set $G\subset \mathbb{C}$ and every analytic function $f:G\to X$  such that $(T-\lambda I) f(\lambda)\equiv 0$ on $G$, one has $f\equiv 0$ on $G$. Note that every operator $T$ such that the point spectrum $\sigma_p(T)$ has empty interior has the SVEP.

\smallskip

When $T$ has the SVEP and $x\in H$, is possible to define the \textit{local resolvent} $\rho_T(x)$ of $T$ at $x$ as the union of all open sets $U \subset \C$ such that there exists a unique vector-valued holomorphic function $f_x: U \rightarrow H$ satisfying
\begin{equation}\label{ecuacion resolvente local-}
(T-zI)f_x(z) = x \quad (z\in U).
\end{equation}
The analytic extension of $f_x$ to $\rho_T(x)$ is called \textit{local resolvent function} of $T$ at $x$.
Likewise, the complement of $\rho_T(x)$ is called the \textit{local spectrum} of $T$ at $x$:
$$\sigma_T(x) = \C \setminus \rho_T(x).$$
Among many other properties, the local spectrum always satisfies that $\sigma_T(x) \subseteq \sigma(T)$ for every $x \in X$ and $\sigma_T(0) = \emptyset$.

\smallskip

The concept of local spectrum allows to define \emph{the local spectral subspaces} of an operator. Given a subset $\Omega\subset \mathbb{C}$, the local spectral subspace of $T$ associated to $\Omega$ is
$$H_T (\Omega)= \{x \in H :\; \sigma_T(x) \subseteq  \Omega\}.$$
If $\Omega_1\subset \Omega_2$ then $H_T (\Omega_1)\subset H_T (\Omega_2)$ and $H_T(\Omega) = H_T(\Omega\cap \sigma(T))$.
It is worthy to note that $H_T(\Omega)$ is a linear manifold that is hyperinvariant for $T$ but not necessarily closed even for closed subsets (see \cite[Chapter 2]{AIENA-book} for instance).

\smallskip

We are in position to define the spectral idempotents of an operator, introduced in \cite{FJKP11} in connection with the study of decomposability.

\smallskip

\begin{definition}
Given $T \in \EL(H)$, an idempotent operator $J \in \EL(H)$ is said to be a \textbf{spectral idempotent} for $T$ if:
    \begin{enumerate}
    	\item $J$ belongs to the bicommutant of $T$, that is $J \in \biconm{T}.$
    	\item The range of $J$ is the spectral subspace associated to $\sigma(T_{\mid \ran(J)})$, that is,
    $$\ran(J) = H_T(\sigma(T_{\mid \ran(J)})).$$
    \end{enumerate}
    In such a case, $J$ is denoted by $J_T(\sigma),$ where $\sigma = \sigma(T_{\mid \ran(J)})$.
		\end{definition}

\smallskip

With this definition at hands, we proceed to prove a characterization of certain spectral subspaces associated to $T$ which will be useful to define particular spectral idempotents. In order to state it, we introduce the following notation. Given $\Lambda=(\lambda_n)\subset \mathbb{C}$ any sequence and $A \subset \C$, $N_A$ will stand for the set of positive integers:
$$N_A = \{n \in \N : \lambda_n \in \Lambda\cap A\}.$$
In addition, given an open set $U$, a holomorphic map $g$ on $U$ and $w \in U$, we define
$$\Gamma(g)(z,w) = \left\{ \begin{matrix} & \frac{g(z)-g(w)}{z-w} \quad & z\neq w \\ & g'(w) \quad &z=w
\end{matrix}\right. $$
Observe that $\Gamma(g)(z,w)$ is continuous in $U\times U$ and for every $w \in U$, the map $z\mapsto \Gamma(g)(z,w)$ is holomorphic in $U$.

\smallskip	

\begin{theorem}\label{caracterizacion subespacios espectrales}
Let $\Lambda=(\lambda_n)_{n\geq 1}\subset \mathbb{C}$ be a bounded sequence and  $u_k = \sumn \al_n^{(k)}e_n$, $v_k = \sumn \beta_n^{(k)}e_n$, $k\geq 1$,  non-zero vectors in $H$
such that $\sumk (\norm{u_k}^2+ \norm{v_k}^2) < \infty.$
Suppose that $T=D_\Lambda + \sumk u_k\otimes v_k  \in\EL(H)$ is a non-scalar operator having the SVEP and $F\subset \C$ a closed set such that $\sigma(T)\cap F \neq \emptyset.$ A vector $x\in H$ belongs to the spectral subspace $H_T(F)$ if and only if there exist a sequence $\{g_k\}_{k\geq 1}$ of unique analytic functions defined on $F^c=\mathbb{C}\setminus F$ satisfying:
		\begin{enumerate}
			\item [(i)] If $x = \sumn x_ne_n$, then $$ x_n = \sumk g_k(\lambda_n)\al_n^{(k)}$$ for every $n \in N_{F^c}.$
			\item [(ii)] The function $$ z\in F^c \mapsto \sum_{n \in N_{F^c}} \left( \sumk \Gamma(g_k)(z,\lambda_n)\al_n^{(k)} \right) e_n$$ is a vector valued analytic function on $F^c$.
			\item [(iii)] For each $k \in \N$, the identity
\begin{equation*}
				\begin{split}
					\sum_{n \in N_F} \frac{x_n \overline{\beta_n^{(k)}}}{\lambda_n-z} & = g_k(z) \left( \sum_{n \in N_F}\frac{\al_n^{(k)}\overline{\beta_n^{(k)}}}{\lambda_n-z}+1\right) - \sum_{n \in N_{F^c}} \left( \summ \Gamma(g_m)(z,\lambda_n)\al_n^{(m)} \right) \overline{\beta_n^{(k)}} \\ & + \sum_{n \in N_F} \left(\sum_{m \neq k} \frac{g_m(z)\al_n^{(m)}\overline{\beta_n^{(k)}}}{\lambda_n-z} \right)
				\end{split}
			\end{equation*}
holds for every $z \in F^c.$
		\item [(iv)] The series $$ \sumn \left| \sumk g_k(z)\al_n^{(k)}\right|^2 $$ is finite for every $z \in F^c.$
		\end{enumerate}
		In such a case, the unique local resolvent function of $T$ at $x$ is given by $$f_x(z) = \sum_{n\in N_F} \frac{x_n-\sumk g_k(z)\al_n^{(k)}}{\lambda_n-z} e_n + \sum_{n\in N_{F^c}} \sumk \Gamma(g_k)(z,\lambda_n)\al_n^{(k)}e_n, \qquad (z \in F^c).$$
	\end{theorem}

\smallskip

The proof of Theorem \ref{caracterizacion subespacios espectrales} is based on that of Theorem 2.1 in \cite{GG} with suitable modifications. We detail it for the sake of completeness.

\smallskip

\begin{proof} Let us start by proving that the conditions $(i)-(iv)$ are necessary. Let $x \in H_T(F)$ be fixed. Since $T$ has the SVEP, there exists an unique vector-valued analytic function $f_x : F^c\rightarrow H$ such that
\begin{equation}\label{ecuacion resolvente local}
		(T-zI)f_x(z) = x
	\end{equation}
for every $z\in F^c.$ Now,
$$(T-zI)f_x(z) = (D_\Lambda-zI)f_x(z) + \sumk \pe{f_x(z),v_k}u_k.$$
We define
$$g_k(z) := \pe{f_x(z),v_k}, \qquad (z\in F^c)$$
for each $k\in \N.$ Obviously, $g_k$ is an analytic function in $F^c.$ Moreover, since $f_x$ unique, it follows that the functions $g_k$ are unique for each $k\in \N.$ Now, write $$f_{x,n}(z) := \pe{f_x(z),e_n}.$$
From \eqref{ecuacion resolvente local} it follows
$$(\lambda_n-z)f_{x,n}(z) + \sumk g_k(z)\al_n^{(k)}=x_n$$
for every $n\in \N$ and $z \in F^c.$ Then, for every $n\in N_{F^c},$
$$\sumk g_k(\lambda_n)\al_n^{(k)} = x_n,$$ what proves $(i)$.

To show $(ii)$, equation \eqref{ecuacion resolvente local} along with $(i)$ yields that
$$f_{x,n}(z) = \sumk \Gamma(g_k)(z,\lambda_n)\al_n^{(k)}$$
for every $n\in N_{F^c}$ and $z \in F^c.$ On the other hand, for every $n \in N_F$
$$f_{x,n}(z) = \frac{x_n-\sumk g_k(z)\al_n^{(k)}}{\lambda_n-z}$$
for every $z\in F^c.$ Accordingly,
\begin{equation}\label{expresion funcion resolvente}
		f_x(z) = \sum_{n\in N_F} \frac{x_n-\sumk g_k(z)\al_n^{(k)}}{\lambda_n-z} e_n + \sum_{n\in N_{F^c}} \sumk \Gamma(g_k)(z,\lambda_n)\al_n^{(k)}e_n.
	\end{equation}

\smallskip

Observe that the map
	$$z\in F^c \mapsto \sum_{n \in N_{F^c}} \left( \sumk \Gamma(g_k)(z,\lambda_n)\al_n^{(k)} \right) e_n$$ is just the orthogonal projection of $f_x$ onto $\overline{\Span \{e_n : n \in N_{F^c}\}},$ so the analyticity of $f_x$ yields that condition $(ii)$ holds.

\smallskip	

To prove $(iii)$, it suffices to recall that $g_k(z) = \pe{f_x(z),v_k}$ for each $k \in \N$ and use the expression \eqref{expresion funcion resolvente} to deduce the equations.
	
\smallskip

Finally, in order to show $(iv)$, let $z \in F^c.$ By recalling that $g_k(z) = \pe{f_x(z),v_k}$ and upon applying twice the Cauchy-Scwartz inequality we obtain
\begin{equation*}
\begin{split}
\sumn \left| \sumk g_k(z)\al_n^{(k)}\right|^2  &
 \leq \sumn  \left( \sumk |g_k(z)|^2 \right) \left( \sumk |\al_n^{(k)}|^2\right)
 \\&= \left(\sumn \sumk |\al_n^{(k)}|^2\right) \left( \sumk |\pe{f_x(z),v_k}|^2 \right) \\ & \leq \left(\sumk \norm{u_k}^2\right) \left( \norm{f_x(z)}^2 \sumk \norm{v_k}^2\right)
  \\& < \infty,
\end{split}
\end{equation*}
which shows $(iv).$
	
\smallskip

In order to prove the sufficiency, let us assume that there exists a sequence $(g_k)_{k\geq 1}$ of analytic functions on $F^c$ satisfying $(i),(ii)$, $(iii)$ and $(iv)$ and  show that $x\in H_T(F)$. Let us define
$$f_x(z) = \sum_{n\in N_F} \frac{x_n-\sumk g_k(z)\al_n^{(k)}}{\lambda_n-z} e_n + \sum_{n\in N_{F^c}} \sumk \Gamma(g_k)(z,\lambda_n)\al_n^{(k)}e_n \qquad (z\in F^c).$$
Observe that, by $(ii)$ and $(iv)$, $f_x$ is analytic on $F^c.$ Moreover, by $(iii)$, it follows that $\pe{f_x(z),v_k}=g_k(z)$ for every $z \in F^c.$

\smallskip

Let us prove that $$(T-zI)f_x(z) = x$$ for every $z \in F^c.$ First, we observe that for every $z \in F^c\setminus \{\lambda_n : n \in F^c\}$
\begin{eqnarray}\label{ecuacion 2}
			(T-zI)f_x(z) & = & (D_\Lambda-z)f_x(z) + \sumk g_k(z)u_k \nonumber
			\\ & = & \sum_{n\in N_F} (x_n - \sumk g_k(z)\al_n^{(k)})e_n + \sum_{n\in N_{F^c}} \sumk (g_k(\lambda_n)-g_k(z))\al_n^{(k)}e_n +   \sumk g_k(z)u_k \nonumber
			\\ & = & x - \sumn \sumk g_k(z)\al_n^{(k)}e_n + \sumk g_k(z)u_k.
\end{eqnarray}
Now, condition $(iv)$ implies that
$$G(z):=\sumn (\sumk g_k(z)\al_n^{(k)}) e_n$$
converges in norm for every $z \in F^c\setminus \{\lambda_n : n \in F^c\}$. In addition, $\sumk g_k(z)u_k$ converges for every $z\in F^c\setminus \{\lambda_n : n \in F^c\}$ due to the boundedness of the operator $\sumk u_k\otimes v_k$ and the equality $\pe{f_x(z), v_k}=g_k(z)$. The equality $\pe{G(z),e_n}=\pe{\sumk g_k(z)u_k, e_n}$ for every $z\in F^c\setminus \{\lambda_n : n \in F^c\}$ and $n\in \N$ shows that
$$G(z)= \sumk g_k(z)u_k$$ for every $z\in F^c\setminus \{\lambda_n : n \in F^c\}$. This along with \eqref{ecuacion 2}, yields that $(T-zI)f_x(z) = x$ for every $z\in F^c\setminus\{ \lambda_n : n \in N_{F^c}\}$.
Since  $\{ \lambda_n : n \in N_{F^c}\}$ is countable, the Identity Theorem yields that $(T-zI)f_x(z) = x$ for every $z \in F^c,$ as we wish to show.
\end{proof}

\smallskip

\begin{remark}
It is important to note that the hypothesis $\sumk (\norm{u_k}^2+\norm{v_k}^2) < \infty$ in Theorem \ref{caracterizacion subespacios espectrales} is only required to demonstrate the necessity of the condition $(iv)$. Clearly, this summability hypothesis implies, in particular, that the operator $K=\sumk u_k\otimes v_k$ is trace class. Indeed, in the absence of this hypothesis, conditions $(i)-(iv)$ suffice to establish that $x\in H_T(F)$, independent of any further convergence assumptions on the norms of $u_k$ and $v_k$, apart from the boundedness of the operator $K$.
\end{remark}
We close this section with a generalization of \cite[Lemma 2.2]{GG2} in the setting of doubly indexed series:

\begin{lemma}\label{lema condicion logaritmo}
		Let $ \Lambda = (\lambda_n)_{n\geq 1}$ be a bounded sequence of complex numbers and let $(\al_n^{(k)})_{n\geq 1,k\geq 1}$ be a doubly indexed sequence of complex numbers satisfying
		\begin{equation}\label{condicion lema}
\sum_{(n,k)\in \mathcal{N}_u} |\al_n^{(k)}|^2\log \left(1+ \frac{1}{|\al_n^{(k)}|}\right) < \infty
		\end{equation}
where $\mathcal{N}_u = \{(n,k) \in \N\times \N : \al_n^{(k)} \neq 0 \}$. Then, for almost every $\x\in \R,$
		$$ \sumn \sumk \frac{|\al_n^{(k)}|^2}{|\PR(\lambda_n)-\x|}<\infty.$$
\end{lemma}

The proof  is similar to that  of \cite[Lemma 3.1]{Klaja} in which  \cite[Lemma 2.1]{FX12} is generalized. The same ideas work to extend \cite[Lemma 2.2]{GG2} to this setting, taking into account that \eqref{condicion lema} implies that both series $\displaystyle \sumn \sumk |\al_n^{(k)}|^2$ and $\displaystyle \sum_{(n,k)\in \mathcal{N}_u} |\al_n^{(k)}|^2\log\left( \frac{1}{|\al_n^{(k)}|}\right)$ are finite, so we omit the proof for the sake of brevity.


\section{Decomposability set and auxiliary operators}\label{seccion 3}
	
In this section, we prove those key results which will allow us to define the spectral idempotents associated to the trace-class perturbations of diagonalizable normal operators covered by the Main Theorem. We start introducing the \textit{decomposability set} of such perturbations.

\smallskip

\begin{definition}
Let $\Lambda = (\lambda_n)_{n\geq 1}\subset \mathbb{C}$ be a bounded sequence not lying in any vertical line and such that the set of accumulation points $\Lambda'$ is not a singleton and denote by
$$a= \min\limits_{z \in \Lambda'} \PR(z) \qquad b=\max\limits_{z \in \Lambda'} \PR(z).$$
Let $\{u_k\}_{k\geq 1}$ and $\{v_k\}_{k\geq 1}$  be non zero vectors in $H$ satisfying \eqref{condicion} and assume that the trace-class perturbation  $T = D_\Lambda + \sumk u_k\otimes v_k$, acting on $H$ satisfies that $\sigma_p(T)\cup\sigma_p(T^*)$ is at most countable. The decomposability set of $T$ consists of those real numbers satisfying:

$$\Delta(T)= \left\{ \x \in (a,b)\setminus \PR(\sigma_p(T)\cup\sigma_p(T^*)) : \sumn \sumk \left(\frac{|\al_n^{(k)}|^2}{|\PR(\lambda_n)-\x|} + \frac{|\beta_n^{(k)}|^2}{|\PR(\lambda_n)-\x|}\right) < \infty\right\}.$$
\end{definition}

\medskip

\begin{remark}
Note that the assumption that the sequence $\Lambda$ does not lie in a vertical line is not a restriction in the class of operators considered, since it can be achieved by a rotation. It is included in order that $a<b$.
\end{remark}

\medskip

Before going further, observe that $\Delta(T)$ is non-empty, and even more, it contains almost every point of $(a,b)$ as a consequence of Lemma \ref{lema condicion logaritmo}.
Likewise, it is clear that the decomposability set of $T$ does not contain the real part of any eigenvalue of $D_\Lambda$, $T$ or $T^*.$

\medskip

To avoid some technicalities, we will make the following assumption in the rest of the manuscript, using the notation introduced in the Main Theorem.

\vspace*{0,3cm}
	
\noindent\framebox{
\vspace*{0,3cm}
\textbf{$(\star)$   Assumption:} \emph{We suppose that both $\sigma(D_\Lambda)=\overline{\Lambda}$ and
		$\sigma(T)$ are contained in the unit disc $\D$.}	
\vspace*{0,3cm}}

\vspace*{0,3cm}
	
\noindent Note that multiplying by some appropriate complex number the assumption is achievable and harmless regarding the existence of non-trivial closed  hyperinvariant subspaces. Likewise, it does not affect to the existence of non trivial spectral idempotents in the bicommutant of $T$.

\smallskip

\smallskip

Now, as in \cite{GG3}, given $\x$ in the real decomposability set  we consider the positively oriented curves
$$\gamma_\x^+:=\ell_\x \cup A_\x^+, \quad \gamma_\x^-:=\ell_\x\cup A_\x^-,$$
where
$$\ell_\x := \{ z \in \overline{\D}: \PR(z) = \x\},$$
and
$$A_\x^+ := \{e^{i\theta} \in \T: \PR(e^{i\theta}) \geq \x\},\quad A_\x^- := \{e^{i\theta} \in \T: \PR(e^{i\theta}) \leq \x\}.$$
Here $\T$ denotes the unit circle. Observe that, by definition, $\gamma_\x^+$ and $\gamma_\x^-$ are piecewise differentiable curves that do not intersect the sequence
$\Lambda$ and the point spectra $\sigma_p(T)$ and $\sigma_p(T^*)$.
	
\noindent Let us denote by
	$$F_\x^+ := \overline{\inte(\gamma_\x^+) }, \quad F_\x^- := \overline{\inte(\gamma_\x^-) },$$
	where $\inte(\gamma_\x^{+})$ and $\inte(\gamma_\x^{-})$ stand for the set of points in $\mathbb{C}$ with index 1 with respect to the closed Jordan curve $\gamma_\x^{+}$ or $\gamma_\x^-,$ respectively.
	
\medskip

In order to define particular square roots of bounded diagonal operators, we fix the square root in $\C\setminus\{0\}$:
$$z \in \C\setminus \{0\} \mapsto z^{1/2} := |z|^{1/2}e^{i\textnormal{arg}(z)/2}$$
where $\arg(z) \in [-\pi,\pi)$. It is important to remark that this is not a continuous branch of the square root, as there are no branches of the square root that are continuous in such a domain. Nevertheless, the chosen square root is \textit{well defined} as a function in $\C\setminus \{0\}$.

Having in mind that $D_\Lambda = \sumn \lambda_n e_n\otimes e_n \in \EL(H)$, we denote
\begin{eqnarray}\label{raiz positiva}
(D_\Lambda-zI)^{1/2} := \sumn (\lambda_n-z)^{1/2}e_n \otimes e_n \qquad  \text{for } z\in \C\setminus \Lambda, \\
(D_\Lambda-zI)^{-1/2}:= \sumn \frac{1}{(\lambda_n-z)^{1/2}}e_n \otimes e_n  \quad  \text{for } z\in \C\setminus \Lambda.
\end{eqnarray}
Observe that $(D_\Lambda-zI)^{1/2}$ is a well defined bounded operator for every $z \in \C\setminus \Lambda$, while $(D_\Lambda-zI)^{-1/2}$ is bounded if $z \notin \overline{\Lambda}$, and it is well-defined but unbounded if $z \in \overline{\Lambda}\setminus \Lambda.$

\smallskip

We are in position to introduce the infinite-dimensional counterpart of those operators considered by Fang and Xia in the proof of \cite[Lemma 3.1]{FX12}.

\begin{proposition} \label{proposicion operadores X,Y}
Let $\Lambda = (\lambda_n)_{n\geq 1}\subset \mathbb{C}$ be a bounded sequence not lying in any vertical line such that $\Lambda'$ is not a singleton and $\{u_k\}_{k\geq 1}$ and $\{v_k\}_{k\geq 1}$  non zero vectors in $H$ satisfying \eqref{condicion}. Assume that the trace-class perturbation  $T = D_\Lambda + \sumk u_k\otimes v_k$ satisfies both $\sigma(D_\Lambda)=\overline{\Lambda}$ and $\sigma(T)$ are contained in $\D$. For each $\x \in \Delta(T)$ the operators
\begin{eqnarray} \label{definicion X,Y}
X^+(z) = \sumk (D_\Lambda-zI)^{-1/2}u_k\otimes  e_k, \qquad \text{for } z\in \gamma_\x^+; \nonumber\\
\qquad \\
Y^+(z) = \sumk e_k\otimes  (D^*_\Lambda-\overline{z}I)^{-1/2}v_k, \qquad \text{for } z \in \gamma_\x^+, \nonumber
\end{eqnarray}
are well-defined bounded operators in $H$. The analogous statement holds for $\gamma_\x^-$.
\end{proposition}
	
\smallskip

\begin{proof}
Let us fix $\x \in \Delta(T)$ and $z \in \gamma_\x^+=\ell_\x\cup A_\x^+$  and prove the result for $X^+(z)$. The proof for $Y^+(z)$ is similar arguing with the adjoint operator  $Y^+(z)^*.$

\smallskip

First, let us show that $(D_\Lambda-zI)^{-1/2}u_k$ belong to $H$ for every $k\geq 1$.  Since $\overline{\Lambda} \subseteq \D$, the operator $(D_\Lambda-zI)^{-1/2}$ is bounded whenever $z\in A_\x^+,$ so $(D_\Lambda-zI)^{-1/2}u_k\in H$. Accordingly, the real task is when $z\in \ell_\x$ and in such a case, having in mind that $\ell_{\x}\cap \Lambda=\emptyset$, we note that

 \begin{equation}\label{acotacion parte real}
 	\norm{(D_\Lambda-zI)^{-1/2}u_k}^2= \sumn \frac{|\al_n^{(k)}|^2}{|\lambda_n-z|} \leq \sumn \frac{|\al_n^{(k)}|^2}{|\PR(\lambda_n)-\x|}<\infty,
 \end{equation}
 so $(D_\Lambda-zI)^{-1/2}u_k \in H$ as well for every $k\geq 1$.

 \medskip

 Now, we deal with the operator $X^+(z)$. As before, if $z\in A_\x^+$ the proof is obvious, so let us assume $z\in \ell_\x$. If $x=\sumn x_ne_n \in H$, applying Cauchy-Schwarz inequality, we have
 \begin{equation*}
 \begin{split}
 	\norm{X^+(z)x}^2& = \norm{\sumk x_k(D_\Lambda-zI)^{-1/2}u_k}^2 = \norm{\sumn \left(\sumk x_k \frac{\al_n^{(k)}}{(\lambda_n-z)^{1/2}}\right)e_n}^2
 	\\ & = \sumn \left| \sumk x_k \frac{\al_n^{(k)}}{(\lambda_n-z)^{1/2}}\right|^2 \leq \norm{x}^2 \sumn \sumk \frac{|\al_n^{(k)}|^2}{|\PR(\lambda_n)-\x|},
 \end{split}
 \end{equation*}
which is finite since $\x$ belongs to the decomposability set $\Delta(T)$.
As a consequence, $X^+(z)$ is a well-defined bounded operator in $H$, which concludes the proof
\end{proof}

\medskip

\medskip
	
\noindent \emph{A word about notation.} For the sake of simplicity, throughout the rest of text we will state most of the results only for the curves $\gamma_\x^+$ and denote the operators in \eqref{definicion X,Y} by $X(z)$ and $Y(z)$, respectively. Clearly all the  results will also hold when the curves $\gamma_\x^-$ are considered.

\medskip

\medskip

Our next goal is proving that $I+X(z)Y(z)$ and $I+Y(z)X(z)$ are invertible for every $z\in \gamma_\x^+.$ For such a purpose, we recall that two operators $T,S \in \EL(H)$ are said to be \textit{quasisimilar} if there exist operators $A,B \in \EL(H)$ such that $TA = AS$ and $BT = SB$ and their kernels satisfy
$$\ker A = \ker A^* = \ker B = \ker B^*=\{0\}.$$
In this case, the operators $A, B$ are called \textit{quasiaffinities}. It is worthy to note that if $T$ and $S$ are quasisimilar, $\sigma_p(T)=\sigma_p(S).$

Next lemma is a generalization of \cite[Theorem 2.5]{FJKP07} and \cite[Theorem 4.8]{GG2}.

\begin{lemma}\label{lema cuasisimilares}
Under the hypotheses of Proposition \ref{proposicion operadores X,Y}, let $\x \in \Delta(T)$ and $\xi \in \gamma_\x^+$. Then, the linear bounded operators $T-\xi I$ and
$$\tilde{T} = (D_\Lambda-\xi I) + \sumk ((D_\Lambda-\xi I)^{-1/2}u_k)\otimes ((D_\Lambda-\xi I)^{1/2})^* v_k$$
are quasisimilar.
\end{lemma}

The proof follows the same lines of that of \cite[Theorem 4.8]{GG2}, and we include it for the sake of completeness.

\begin{proof} Note that both operators $(D_\Lambda-\xi I)^{1/2}$ and $T-\xi I= (D_\Lambda-\xi I)+ \sumk u_k\otimes v_k$ are quasiaffinities, since the curves $\gamma_\x^+$ and $\gamma_\x^-$ do not contain any eigenvalue of $D_\Lambda$, $T$ or $T^*$.
Now, define
$$U= (D_\Lambda-\xi I)^{-1/2}(T-\xi I)$$
and observe that it is also a quasiaffinity. Since, $(T-\xi I) = (D_\Lambda-\xi I)^{1/2}U$ and $\tilde{T}=  U(D_\Lambda-\xi I)^{1/2},$ we deduce
$$(T-\xi I)(D_\Lambda-\xi I)^{1/2} = (D_\Lambda-\xi I)^{1/2}\tilde{T}.$$
Finally, $U(T-\xi I) = U(D_\Lambda -\xi I)^{1/2}U = \tilde{T} U$, and the statement follows.
\end{proof}
	
We are in position to state a first step to approach the proof of the Main Theorem.

\begin{theorem}\label{proposicion invertibles}
Under the hypotheses of Proposition \ref{proposicion operadores X,Y}, if $\x$ is in the decomposability set $\Delta(T)$ and $z \in \gamma_\x^+$, the operators
$$I+X(z)Y(z), \qquad I+Y(z)X(z)$$
are invertible in $H$.
\end{theorem}

\begin{proof} Let $\x\in \Delta(T)$ and $z \in \gamma_\x^+=\ell_\x\cup A_\x^+$ be fixed. First, note that
$$X(z)Y(z) = \sumk  (D_\Lambda-zI)^{-1/2}u_k\otimes   (D_\Lambda-zI)^{-1/2} )^*v_k.$$
Observe that if $z$ lies in the segment $\ell_\x$

	\begin{equation*}
		\begin{split}
&\sumk \norm{(D_\Lambda-zI)^{-1/2}u_k}\norm{(D_\Lambda-zI)^{-1/2})^* v_k}
\\ & \leq \left(\sumk \norm{(D_\Lambda-zI)^{-1/2}u_k}^2 \right)^{1/2}\left(\sumk \norm{(D_\Lambda-zI)^{-1/2})^* v_k}^2 \right)^{1/2}
\\ & \leq \left( \sumk \sumn \frac{|\al_n^{(k)}|^2}{|\PR(\lambda_n)-\x|}\right)^{1/2}\left( \sumk \sumn \frac{|\beta_n^{(k)}|^2}{|\PR(\lambda_n)-\x|}\right)^{1/2}
\\ &< \infty,
		\end{split}
	\end{equation*}
and, accordingly,  $X(z)Y(z)$ is a trace-class operator.

On the other hand, having in mind that $(D_\Lambda-zI)^{-1/2}$ is  uniformly bounded when $z \in A_\x^+$ and that the operator $K=\sumk u_k\otimes v_k$ is trace-class, it follows that  $X(z)Y(z)$ is as well trace class whenever $z\in A_\x^+$. Consequently $X(z)Y(z)$ is a trace class operator, and hence compact.

\smallskip
	
In order to prove that $I+X(z)Y(z)$ is invertible, let us argue by contradiction. Assume that $I+X(z)Y(z)$ is not invertible. Hence, by compactness, it follows that $-1\in\sigma_p(X(z)Y(z))$ and $-1 \in \sigma_p( (X(z)Y(z))^*).$ That is, there exists $h \in H\setminus\{0\}$ such that
$$h+ (X(z)Y(z))^*h = 0.$$ Applying $(D_\Lambda-z I)^*,$ we deduce that
$$(D_\Lambda - z I)^*h + \left( \sumk ((D_\Lambda-z I)^{1/2})^*v_k \otimes (D_\Lambda - z I)^{-1/2}u_k \right)h=0.$$
Then, $0$ is an eigenvalue of the operator
$$S= (D_\Lambda-z I)^* +  \sumk ((D_\Lambda-z I)^{1/2})^*v_k \otimes (D_\Lambda - z I)^{-1/2}u_k.$$
By Lemma \ref{lema cuasisimilares}, it follows that $T-z I$ is quasisimilar to $S^*$, so $T^*-\overline{z} I$ is quasisimilar to $S$.
Finally, $\overline{z}$ is an eigenvalue of $T^*$, which is a contradiction.

\medskip
Hence,  $I+X(z)Y(z)$ is invertible, which deals with the first statement of the theorem \ref{proposicion invertibles}.
	
Nevertheless, from here the invertibility of $I+Y(z)X(z)$ is just a standard fact (see \cite[p. 199]{Conway}) since the operator $I-Y(z)(I+X(z)Y(z))^{-1}X(z)$ is an inverse for $I+Y(z)X(z).$ This concludes the proof.	
\end{proof}

\medskip

As a consequence of Theorem \ref{proposicion invertibles}, we observe that if $\x \in \Delta(T)$ and $z\in \gamma_\x^+$ we may write
\begin{equation}\label{definicion inverso}
(I+Y(z)X(z))^{-1} = \sumi \sumj a_{i,j}(z)e_i\otimes e_j
\end{equation}
where
\begin{equation}\label{definicion a_{i,j}}
a_{i,j}(z) = \pe{(I+Y(z)X(z))^{-1}e_j,e_i}.
\end{equation}

\medskip

Our  next step consists of  relating the coefficients $a_{i,j}(z)$, by means of a functional equation, to the operator $T=D_\Lambda+\sumk u_k\otimes v_k$ which allows us to construct the idempotents announced in the outline of the proof in introductory section. For such a purpose, we recall in this setting the \textit{Borel series} associated to $T$, that is, the series defined by
\begin{equation}\label{serie de Borel}
f_T^{(i,j)}(z) = \sumn \frac{\al_n^{(i)}\overline{\beta_n^{(j)}}}{\lambda_n-z},\qquad  ((i,j)\in \N\times \N)
\end{equation}
for those $z\in \C$ such that the  series converges. It is worthy to remark that $f_T^{(i,j)}$ defines an analytic function in $\C\setminus \overline{\Lambda}$, and played a fundamental role in the existence of the invariant subspaces in the finite-rank perturbation case (see the works \cite{GG,GG2,GG3}).

\medskip

Next theorem is the main result of this section, and the key of the construction of the idempotents to prove the Main Theorem.

\begin{theorem}\label{teorema formulas}
Under the hypotheses of Proposition \ref{proposicion operadores X,Y},  let $\x\in \Delta(T)$ and $z \in \gamma_\x^+$. There exists a positive constant $C_\x^+>0$, independent of $z$,  such that for every $x=\sumn x_ne_n\in H$
\begin{equation}\label{acotacion condicional}
\sum_{i=1}^\infty \Big| \sum_{j=1}^\infty x_ja_{i,j}(z)\Big|^2 \leq C_\x^+ \norm{x}^2.
\end{equation}
Moreover, for each $n \in \N$
\begin{equation}\label{alien cofactor}
	\sum_{k=1}^\infty \sum_{j=1}^\infty x_ja_{k,j}(z)(\delta_{k,n}+f_T^{(k,n)}(z)) = x_n,
\end{equation}
where $\delta_{k,n} =1$ if $n=k$ and $0$ otherwise.
\end{theorem}

\medskip

In order to prove Theorem \ref{teorema formulas}, the following fact is required:

\medskip

\begin{proposition}\label{continuidad}
Under the hypotheses of Proposition \ref{proposicion operadores X,Y}, for each $\x \in \Delta(T)$ the map
$$z \in \gamma_\x^+ \mapsto (I+Y(z)X(z))^{-1} \in \EL(H) $$
is continuous in the norm topology of $\EL(H)$.
\end{proposition}

\begin{proof}
Let $\x \in \Delta(T)$  be fixed. Since $I+Y(z)X(z)$ is invertible for every $z\in \gamma_\x^+$  (Theorem \ref{proposicion invertibles}), and taking inverses is continuous in the topology of $\EL(H)$, it suffices to prove that the map
$$z \in \gamma_\x^+ \mapsto Y(z)X(z)\in \EL(H)$$ is continuous.

For every $z \in \gamma_\x^+=\ell_\x\cup A_\x^+$, it is easy to check that
$$ Y(z)X(z) = \sumn \sumk f_T^{(k,n)}(z)e_n\otimes e_k.$$
If $z \in \ell_\x$ and $x=\sumn x_n e_n \in H$, upon applying Cauchy-Schwarz inequality twice we deduce that

\begin{equation*}
		\begin{split}
\norm{Y(z)X(z)x}^2 & = \sumn \left| \sumk x_k f_T^{(k,n)}(z)\right|^2 \\
& \leq \norm{x}^2 \sumn \sumk |f_T^{(k,n)}(z)|^2\\
& = \norm{x}^2 \sumn \sumk \left| \sumj \frac{\al_j^{(k)}\overline{\beta_j^{(n)}}}{\lambda_j-z}\right|^2\\
& \leq \norm{x}^2 \sumn \sumk \left(\sumj \frac{|\al_j^{(k)}|^2}{|\lambda_j-z|}\right)\left(\sumj \frac{|\beta_j^{(n)}|^2}{|\lambda_j-z|}\right) \\
&  = \norm{x}^2 \left( \sumk \sumj \frac{|\al_j^{(k)}|^2}{|\lambda_j-z|}\right)\left( \sumn \sumj \frac{|\beta_j^{(n)}|^2}{|\lambda_j-z|}\right)\\
& \leq \norm{x}^2 \left( \sumk \sumj \frac{|\al_j^{(k)}|^2}{|\PR(\lambda_j)-\x|}\right)\left( \sumn \sumj \frac{|\beta_j^{(n)}|^2}{|\PR(\lambda_j)-\x|}\right),
		\end{split}
	\end{equation*}
which is finite since $\x$ belongs to the decomposability set  $\Delta(T)$.

On the other hand, if $z\in A_\x^+,$ denoting $d:= \textnormal{dist}(\T, \overline{\Lambda}) >0$ and applying the same arguments, we deduce that
$$\norm{Y(z)X(z)}^2 \leq \frac{1}{d^2}  \left(\sumk \sumj |\al_j^{(k)}|^2\right)\left( \sumn \sumj |\beta_j^{(n)}|^2\right),$$
which is also finite.

Note that, as a consequence of the previous estimates, the finite-rank operators
$$U_N( \,\cdot \, ) := \sum_{n=1}^N \sum_{k=1}^N f_T^{(k,n)}(\,\cdot \, )e_n\otimes e_k $$
converge uniformly on $\gamma_\x^+$ to $Y(\,\cdot \,)X(\,\cdot \,).$
Accordingly, it suffices  to show that for each $N \in \N$ the maps
$$z \in \gamma_\x^+ \mapsto U_N(z)$$
are continuous, or equivalently, that the maps
$$z\in \gamma_\x^+ \mapsto f_T^{(k,n)}(z)\in \C$$
are continuous for every $k,n \in \N.$ But taking into account that
$$ \sumj \left| \frac{\al_j^{(k)}\overline{\beta_j^{(n)}}}{\lambda_j-z}\right| \leq \left( \sumj \frac{|\al_j^{(k)}|^2}{|\PR(\lambda_j)-\x|}\right)^{1/2}\left( \sumj \frac{|\beta_j^{(n)}|^2}{|\PR(\lambda_j)-\x|}\right)^{1/2}$$ for every $z\in \ell_\x$ and
$$\sumj \left| \frac{\al_j^{(k)}\overline{\beta_j^{(n)}}}{\lambda_j-z}\right| \leq \frac{1}{d} \norm{u_k}\norm{v_n}$$
for every $z \in A_\x^+,$ it follows that each Borel series $f_T^{(k,n)}$ converges uniformly  on $\gamma_\x^+,$ so they define continuous functions on $\gamma_\x^+$. This concludes the proof.
\end{proof}

\smallskip
	
Finally, we are in position of proving Theorem \ref{teorema formulas}.

\smallskip

\begin{demode}\emph{Theorem \ref{teorema formulas}.}
Let $\x\in \Delta(T)$ fixed and $z \in \gamma_\x^+$. To prove \eqref{acotacion condicional}, having in mind the expression of $a_{i,j}(z)$ given by \eqref{definicion inverso},  we note that for each $x =\sumn x_ne_n \in H$,
\begin{equation}\label{estimacion puntual}
	 		\norm{(I+Y(z)X(z))^{-1}x}^2 = \norm{\sum_{i=1}^\infty \left(\sum_{j=1}^\infty x_ja_{i,j}(z)\right)e_i}^2 = \sumi \Big| \sum_{j=1}^\infty x_ja_{i,j}(z)\Big|^2
\end{equation}
Now the map $z\in \gamma_\x^+ \mapsto (I+Y(z)X(z))^{-1}\in\EL(H)$ is continuous  by Proposition \ref{continuidad}, so  the compactness of $\gamma_\x^+$ yields that there exists $C_\x^+ > 0$ such that
$$\norm{(I+Y(z)X(z))^{-1}}^2\leq C^+_\x$$
for every $z\in\gamma_\x^+$. This along with \eqref{estimacion puntual} yields
$$
\sum_{i=1}^\infty \Big| \sum_{j=1}^\infty x_ja_{i,j}(z)\Big|^2 \leq C_\x^+ \norm{x}^2
$$
which is \eqref{acotacion condicional}.
	 	
\medskip
	 	
Finally, let us prove that \eqref{alien cofactor} holds. Fix $z\in \gamma_\x^+$ and note that
$$I+Y(z)X(z) = \sumn \sumk (\delta_{k,n}+f_T^{(k,n)}(z))e_n\otimes e_k.$$
If $x=\sumn x_ne_n \in H$, a computation shows
	 	\begin{equation*}
	 	\begin{split}
	 	\sumn x_ne_n &= (I+Y(z)X(z))(I+Y(z)X(z))^{-1}x
	 	\\& = (I+Y(z)X(z))\left( \sumi \left(\sumj x_ja_{i,j}(z)\right)e_i\right)
	 	\\ & =  \sumn \sumk (\delta_{k,n}+f_T^{(k,n)}(z)) \left \langle  \sumi \left(\sumj x_ja_{i,j}(z)\right)e_i, e_k \right \rangle e_n
	 	\\ & = \sumn \sumk (\delta_{k,n}+f_T^{(k,n)}(z)) \sumj x_ja_{k,j}(z) e_n
	 	\\ & = \sumn \left( \sumk \sumj x_ja_{k,j}(z)(\delta_{k,n}+f_T^{(k,n)}(z))\right)  e_n.
	 	\end{split}
	 	\end{equation*}
Matching the $n-$th Fourier coefficients for $n\in \N$, \eqref{alien cofactor} follows and the proof of the Theorem \ref{teorema formulas} is finished.
\end{demode}

\medskip

	\section{Spectral idempotents: the key ingredient}\label{seccion 4}

In this section we introduce a family of non-trivial idempotents  in the bicommutant $\biconm{T}$ such that their ranges  are spectral subspaces for $T$.

\begin{theorem}\label{idempotentes acotados}
Let $\Lambda = (\lambda_n)_{n\geq 1}\subset \mathbb{C}$ be a bounded sequence not lying in any vertical line such that $\Lambda'$ is not a singleton and $\{u_k\}_{k\geq 1}$ and $\{v_k\}_{k\geq 1}$ are  non zero vectors in $H$ satisfying \eqref{condicion}. Assume that the trace-class perturbation  $T = D_\Lambda + \sumk u_k\otimes v_k$ satisfies that both  $\sigma(D_\Lambda)=\overline{\Lambda}$ and $\sigma(T)$ are contained in  $\D$. Let $\x \in \Delta(T)$ and $X(\, \cdot \,)$, $Y(\, \cdot \,)$ the operators defined on $\gamma_\x^+$ given in \eqref{definicion X,Y} and write
$$
(I+Y(\, \cdot \, )X(\, \cdot \, ))^{-1} = \sumi \sumj a_{i,j}(\, \cdot \, )e_i\otimes e_j.
$$
Then the operator
\begin{equation}\label{definicion idempotentes}
J_\x^{+} x = \sum_{n \in N_{F_\x^+}} x_ne_n + \sumn \sum_{k=1}^\infty \left( \frac{1}{2\pi i} \int_{\gamma_\x^{+}} \frac{ \sumj \frac{x_j}{\lambda_j-\xi}\left(\summ \overline{\beta_j^{(m)}}a_{k,m}(\xi)\right)}{\lambda_n-\xi}        d\xi    \right)\al_n^{(k)}e_n,
\end{equation}
where $x=\sumn x_ne_n \in H$, is well-defined and bounded in $H$.
\end{theorem}

As we mentioned in Section \ref{seccion 3},  for the sake of simplicity, we will deal with the curves $\gamma_\x^+$ and the analogous results will also hold when the curves $\gamma_\x^-$ are considered.

\begin{proof}
It is obvious that the map $x=\sumn x_ne_n\in H \mapsto \sum_{n\in N_{F_\x^+}}x_ne_n$ is bounded, so our task is dealing with the second summand of \eqref{definicion idempotentes}.
Namely, the goal is finding a bound for

\begin{equation}\label{acotacion J1}
\sumn \left| \sum_{k=1}^\infty \left( \int_{\gamma_\x^+} \frac{ \sumj \frac{x_j}{\lambda_j-\xi}\left(\summ \overline{\beta_j^{(m)}}a_{k,m}(\xi)\right)}{\lambda_n-\xi}d\xi    \right)\al_n^{(k)}\right|^2
\end{equation}
which is, clearly, less than or equal to
\begin{equation}\label{acotacion J}
\sumn \left(\sum_{k=1}^\infty  \sumj |x_j| \int_{\gamma_\x^+} \frac{ \left|\summ \overline{\beta_j^{(m)}}a_{k,m}(\xi)\right|}{|\lambda_j-\xi||\lambda_n-\xi|}|d\xi|     \left|\al_n^{(k)}\right|\right)^2.
\end{equation}
				
Now, we apply the Cauchy-Schwarz inequality twice, first under the integral sign and then to series indexed in $k\geq 1$, having that \eqref{acotacion J} is less than or equal to

\begin{equation*}
\begin{split}
& \sumn \left(\sum_{k=1}^\infty  \sumj |x_j| \left(  \int_{\gamma_\x^+} \frac{ \left|\summ \overline{\beta_j^{(m)}}a_{k,m}(\xi)\right|^2}{|\lambda_j-\xi|^2} |d\xi| \right)^{1/2}\left( \int_{\gamma_\x^+}\frac{|d\xi|}{|\lambda_n-\xi|^2}\right)^{1/2}  \left|\al_n^{(k)}\right|   \right )^2  \leq \\
& \sumn \left( \sumk \left| \sumj |x_j| \left(  \int_{\gamma_\x^+} \frac{ \left|\summ \overline{\beta_j^{(m)}}a_{k,m}(\xi)\right|^2}{|\lambda_j-\xi|^2} |d\xi| \right)^{1/2}\right|^2  \right) \left( \sumk \int_{\gamma_\x^+} \frac{|d\xi|}{|\lambda_n-\xi|^2}\left|\al_n^{(k)}\right|^2\right) 	
\end{split}
\end{equation*}
which is equal to
\begin{equation}\label{acotacion JJ}
\underbrace{\left( \sumk \left| \sumj |x_j| \left(  \int_{\gamma_\x^+} \frac{ \left|\summ \overline{\beta_j^{(m)}}a_{k,m}(\xi)\right|^2}{|\lambda_j-\xi|^2} |d\xi| \right)^{1/2}\right|^2  \right)}_\text{(I)} \cdot  \underbrace{\left(\sumn \sumk \int_{\gamma_\x^+} \frac{|d\xi|}{|\lambda_n-\xi|^2}\left|\al_n^{(k)}\right|^2\right)}_\text{(II)}.
\end{equation}

First, we focus our attention on (II) in \eqref{acotacion JJ}. As it is shown in the proof of \cite[Theorem 2.8]{GG3}, there exists a constant $C>0$ such that
		\begin{equation}\label{integral}
			\int_{\gamma_\x^+} \frac{|d\xi|}{|\lambda_n-\xi|^2} \leq \frac{C}{|\PR(\lambda_n)-\x|}.
		\end{equation}
Thus,
$$\text{(II)}=\sumn \sumk \int_{\gamma_\x^+} \frac{|d\xi|}{|\lambda_n-\xi|^2}\left|\al_n^{(k)}\right|^2 \leq C \sumn \sumk \frac{|\al_n^{(k)}|^2}{|\PR(\lambda_n)-\x|},$$ which converges since $\x \in \Delta(T).$

Now, let us bound the factor (I) in \eqref{acotacion JJ}. Upon applying the Cauchy-Schwarz inequality to the series indexed in $j\geq 1$, we obtain
\begin{eqnarray}\label{acotacion JJJ}
\underbrace{\sumk \left| \sumj |x_j| \left(  \int_{\gamma_\x^+} \frac{ \left|\summ \overline{\beta_j^{(m)}}a_{k,m}(\xi)\right|^2}{|\lambda_j-\xi|^2} |d\xi| \right)^{1/2}\right|^2}_\text{(I)}
& \displaystyle \leq  \norm{x}^2 \sumk \sumj \int_{\gamma_\x^+} \frac{ \left|\summ \overline{\beta_j^{(m)}}a_{k,m}(\xi)\right|^2}{|\lambda_j-\xi|^2} |d\xi| \nonumber \\
& \displaystyle =     \norm{x}^2  \sumj \int_{\gamma_\x^+} \frac{ \sumk \left|\summ \overline{\beta_j^{(m)}}a_{k,m}(\xi)\right|^2}{|\lambda_j-\xi|^2} |d\xi|,
\end{eqnarray}
where last equality follows as consequence of the Monotone Convergence Theorem.
		
At this point, we note that for each $j\geq 1,$ the vector $\summ \overline{\beta_j^{(m)}}e_m \in H$, so by \eqref{acotacion condicional} in Theorem \ref{teorema formulas}, we deduce that
$$ \sumk \left|\summ \overline{\beta_j^{(m)}}a_{k,m}(\xi)\right|^2 \leq C_\x^+ \summ |\beta_j^{(m)}|^2$$
for every $\xi\in \gamma_\x^+.$ This along with \eqref{integral} yields that
$$	\sumj \int_{\gamma_\x^+} \frac{ \sumk \left|\summ \overline{\beta_j^{(m)}}a_{k,m}(\xi)\right|^2}{|\lambda_j-\xi|^2} |d\xi| \leq C\cdot C_\x^+ \sumj \summ \frac{|\beta_j^{(m)|^2}|}{|\PR(\lambda_j)-\x|},
$$
which is finite, since $\x\in \Delta(T)$.

Consequently,  there exists a positive constant $M>0$ given by
		\begin{equation}\label{acotacion norma}
			M:= C^2 C_\x^+\left(\sumj \summ \frac{|\beta_j^{(m)}|^2}{|\PR(\lambda_j)-\x|}\right)\left( \sumn \sumk \frac{|\al_n^{(k)}|^2}{|\PR(\lambda_n)-\x|}\right),
		\end{equation}
such that
$$\norm{J_\x^+ x} \leq \norm{\sum_{n\in N_{F_\x^+}} x_ne_n} + M^{1/2}\norm{x} \leq (1+M^{1/2})\norm{x}, $$
for $x\in H$, which proves the theorem.
	\end{proof}

The next step consists of applying Theorem \ref{caracterizacion subespacios espectrales} to show that the ranges of $J_\x^{\pm}$ are contained in the spectral subspaces associated to the closed sets $F_\x^{\pm}$ for every $\x$ in the decomposability set of $T$. We state it for $J_\x^+$:

\smallskip

\begin{theorem}\label{proposicion rango}
Under the hypotheses of Theorem \ref{idempotentes acotados}, for each $\x \in \Delta(T)$
$$\ran(J_\x^+)\subseteq H_T(F_\x^+).$$
\end{theorem}
	
\smallskip

\begin{proof}
Let $y=\sumn y_n e_n \in H$ be fixed and denote $x= J_\x^+ y$. We will prove that $x \in H_T(F_\x^+)$ applying Theorem \ref{caracterizacion subespacios espectrales}: note that $T$ enjoys the SVEP (because $\sigma_p(T)$ is at most countable) and the series $\sumk (\norm{u_k}^2+\norm{v_k}^2)$ is finite (as a consequence of \eqref{condicion}). Accordingly, our task is defining  holomorphic functions $g_k:\mathbb{C}\setminus F_\x^+ \rightarrow \C$ for $k\geq 1$ satisfying conditions $(i)-(iv)$ of the aforementioned theorem.

\medskip

For $k\geq 1$ let
$$g_k(z) = \frac{1}{2\pi i} \int_{\gamma_\x^+} \frac{ \sumj \frac{y_j}{\lambda_j-\xi}\left(\summ \overline{\beta_j^{(m)}}a_{k,m}(\xi)\right)}{z-\xi} d\xi \qquad (z\in \C\setminus \gamma_\x^+).$$ We will show that $g_k$ is a well-defined and holomorphic function in $\mathbb{C}\setminus F_\x^+$.

Fix $K\subset \C\setminus \gamma_\x^+$ a compact set and denote $d = \textnormal{dist}(K, \gamma_\x^+) > 0.$ Observe that, for every $z \in K,$ Cauchy-Schwarz inequality applied twice yields

\begin{eqnarray}\label{integral bien definida}
\sumj\int_{\gamma_\x^+} \frac{|y_j|\left|\summ \overline{\beta_j^{(m)}}a_{k,m}(\xi)\right|}{|\lambda_j-\xi||z-\xi|}|d\xi|
&\leq & \displaystyle  \frac{\norm{y}}{d} \left( \sumj \left( \int_{\gamma_\x^+} \frac{\left|\summ \overline{\beta_j^{(m)}}a_{k,m}(\xi)\right|}{ |\lambda_j-\xi|} |d\xi|     \right)^2 \right)^{1/2}  \nonumber \\
&\leq & \displaystyle  \frac{\norm{y}}{d} \; \ell\cdot \left( \sumj \int_{\gamma_\x^+} \frac{\left|\summ \overline{\beta_j^{(m)}}a_{k,m}(\xi)\right|^2}{ |\lambda_j-\xi|^2} \right)^{1/2}
\nonumber	 \\
&\leq & \displaystyle \frac{\norm{y}}{d} \; \ell \cdot(C_\x^+)^{1/2}  \underbrace{\left(\sumj \summ |\beta_j^{(m)}|^2 \int_{\gamma_\x^+}\frac{|d\xi|}{|\lambda_j-\xi|^2}     \right)^{1/2}}_\text{(III)},
\end{eqnarray}
where $\ell := \textnormal{length}(\gamma_\x^+)^{1/2}$ and the bound  \eqref{acotacion condicional} of Theorem \ref{teorema formulas} has been applied in the last inequality.


Finally, \eqref{integral} yields that (III) is finite since $\x\in \Delta (T)$, and hence the series in \eqref{integral bien definida}
$$\sumj\int_{\gamma_\x^+} \frac{|y_j|\left|\summ \overline{\beta_j^{(m)}}a_{k,m}(\xi)\right|}{|\lambda_j-\xi||z-\xi|}|d\xi|$$
converges uniformly in $K$. As a consequence, we have that
\begin{equation}\label{igualdad de series}
g_k(z) = \frac{1}{2\pi i} \int_{\gamma_\x^+} \frac{ \sumj \frac{y_j}{\lambda_j-\xi}\left(\summ \overline{\beta_j^{(m)}}a_{k,m}(\xi)\right)}{z-\xi} d\xi= \sumj \frac{y_j}{2\pi i} \int_{\gamma_\x^+} \frac{ \summ \overline{\beta_j^{(m)}}a_{k,m}(\xi)}{(\lambda_j-\xi)(z-\xi)}        d\xi
\end{equation}
for  $z\in \C\setminus \gamma_\x^+$.

Since the series in the right hand side of \eqref{igualdad de series} converges uniformly on compact subsets of $\C\setminus \gamma_\x^+$, in order to prove that $g_k$ is holomorphic it would be enough to show that for each $j \in \N$ the map
\begin{equation}\label{integrales holomorfas}
z\in \C\setminus \gamma_\x^+ \mapsto \int_{\gamma_\x^+} \frac{ \summ \overline{\beta_j^{(m)}}a_{k,m}(\xi)}{(\lambda_j-\xi)(z-\xi)}d\xi
\end{equation}
 is holomorphic.

\noindent For such a task, fix $j\in \N$ and define $\beta_j := \summ \overline{\beta_j^{(m)}}e_m \in H.$ Observe that given $\xi \in \gamma_\x^+,$
$$ \summ \overline{\beta_j^{(m)}}a_{k,m}(\xi) = \pe{(I+Y(\xi)X(\xi))^{-1}\beta_j,e_k},$$
which along with Proposition \ref{continuidad} yield that the map
$$\xi \in \gamma_\x^+ \mapsto \frac{ \summ \overline{\beta_j^{(m)}}a_{k,m}(\xi)}{(\lambda_j-\xi)(z-\xi)}$$
is continuous for every $z\in \C\setminus \gamma_\x^+$.

Finally, for $\xi \in \gamma_\x^+$ the map
$$z\in \C\setminus \gamma_\x^+\mapsto \frac{1}{z-\xi}$$
is holomorphic, so for each $j \in \N$ the map \eqref{integrales holomorfas} is holomorphic and, hence, $g_k$ is holomorphic as well.

\medskip

Now, let us show that for each $k\geq 1$ the function $g_k$ meets the conditions $(i)-(iv)$ of Theorem \ref{caracterizacion subespacios espectrales}. To prove $(i)$, let $n\in N_{(F_\x^+)^c}$ and note that
$$\pe{x,e_n} = \pe{J_\x^+ y,e_n} = \sumk g_k(\lambda_n)\al_n^{(k)}.$$
Thus, condition $(i)$ is fulfilled.

\medskip

In order to show $(ii)$, observe that for every $z\in \C\setminus\gamma_\x^+, n\in \N$ and $z\neq \lambda_n$, we have
 $$\Gamma(g_k)(z,\lambda_n) = \frac{-1}{2\pi i} \int_{\gamma_\x^+} \frac{ \sumj \frac{y_j}{\lambda_j-\xi}\left(\summ \overline{\beta_j^{(m)}}a_{k,m}(\xi)\right)}{(z-\xi)(\lambda_n-\xi)}d\xi.$$

If  $h=\sumn h_ne_n \in H$,  we are required to prove that the map
 \begin{equation}\label{serie condicion (ii)}
z\in \C\setminus \gamma_\x^+ \mapsto \sum_{n \in N_{(F_\x^+)^c}}\sumk \left (\frac{-1}{2\pi i} \int_{\gamma_\x^+} \frac{ \sumj \frac{y_j}{\lambda_j-\xi}\left(\summ \overline{\beta_j^{(m)}}a_{k,m}(\xi)\right)}{(z-\xi)(\lambda_n-\xi)}d\xi \right )\,  \al_n^{(k)}\; \overline{h_n}
\end{equation}
is holomorphic. Since the map $z\in \C\setminus \gamma_\x^+ \mapsto \Gamma(g_k)(z,\lambda_n)\al_n^{(k)}$ is holomorphic for every $n\in N_{(F_\x^+)^c}$ and $k\in\N$, it suffices to show that the double series in \eqref{serie condicion (ii)} converges uniformly on compact subsets of $\C\setminus \gamma_\x^+.$

Fix $K\subset \C\setminus \gamma_\x^+ $ a compact subset and denote $d=\textnormal{dist}(K, \gamma_\x^+).$
Note that
\begin{equation}\label{convergencia serie triple}
\begin{split}
 & \sum_{n \in N_{(F_\x^+)^c}}\sumk \frac{1}{2\pi } \left|\int_{\gamma_\x^+} \frac{ \sumj \frac{y_j}{\lambda_j-\xi}\left(\summ \overline{\beta_j^{(m)}}a_{k,m}(\xi)\right)}{(z-\xi)(\lambda_n-\xi)}d\xi\al_n^{(k)}\overline{h_n}\right|
 \\& \leq \frac{1}{2\pi} \sumn \sumk \sumj |y_j| \int_{\gamma_\x^+} \frac{\left|\summ \overline{\beta_j^{(m)}}a_{k,m}(\xi)\right|}{|\lambda_j-\xi||z-\xi||\lambda_n-\xi|}|d\xi| |\al_n^{(k)}||h_n|
 \\ & \leq \frac{\norm{h}}{2\pi d} \underbrace{\left( \sumn \left| \sumk \sumj |y_j| \int_{\gamma_\x^+} \frac{ \left|\summ \overline{\beta_j^{(m)}}a_{k,m}(\xi)\right| }{|\lambda_j-\xi||\lambda_n-\xi|}|d\xi| |\al_n^{(k)}|\right|^2\right)^{1/2}}_\text{(IV)},
\end{split}
  \end{equation}
where last inequality follows upon applying Cauchy-Schwarz inequality. But, observe that (IV) looks like \eqref{acotacion J} (where $y_j$ plays the role of $x_j$), and \eqref{acotacion J}  was already  shown to be finite. Consequently,  $(ii)$ follows as well.

\medskip

Now, let us prove condition $(iii)$. Having into account that $g_k$ is holomorphic in $\C\setminus \gamma_\x^+,$  our task is proving that
\begin{equation}\label{igualdad condicion (iii)}
\begin{split}
\sum_{n \in N_{F_\x^+}} \frac{y_n\overline{\beta_n^{(k)}}}{\lambda_n-z} & = g_k(z) - \sumn \suml \Gamma(g_{\ell})(z,\lambda_n)\al_n^{(\ell)}\overline{\beta_n^{(k)}},
\end{split}
\end{equation}
for each $z \in (F_\x^+)^c$ and $k\in \N$.

Now,
\begin{equation}\label{condicion iii--}
\sumn \suml \Gamma(g_{\ell})(z,\lambda_n)\al_n^{(\ell)}\overline{\beta_n^{(k)}} = \underbrace{\sumn \suml \frac{-1}{2\pi i} \int_{\gamma_\x^+} \frac{ \sumj \frac{y_j}{\lambda_j-\xi}\left(\summ \overline{\beta_j^{(m)}}a_{\ell,m}(\xi)\right)}{(z-\xi)(\lambda_n-\xi)}d\xi\al_n^{(\ell)}\overline{\beta_n^{(k)}}}_\text{(V)},
\end{equation}
for $z \in (F_\x^+)^c$ and $k\in \N$. But the series
$$\sumn \suml \sumj |y_j| \int_{\gamma_\x^+} \frac{\left|\summ \overline{\beta_j^{(m)}}a_{\ell,m}(\xi)\right|}{|\lambda_j-\xi||z-\xi||\lambda_n-\xi|}|d\xi||\al_n^{(\ell)}||\beta_n^{(k)}|$$
converges (analogously as \eqref{acotacion J}), so we can change the order of summation of $n,\ell,j$ and the integral in (V) (note that we are not changing the series indexed in $m$) and  \eqref{condicion iii--} turns out to be:

\begin{equation*}
\begin{split}
\sumn \suml \Gamma(g_{\ell})(z,\lambda_n)\al_n^{(\ell)}\overline{\beta_n^{(k)}} & =   \frac{-1}{2\pi i} \int_{\gamma_\x^+} \frac{ \sumj \frac{y_j}{\lambda_j-\xi}\suml\left(\summ \overline{\beta_j^{(m)}}a_{\ell,m}(\xi)\right)\sumn  \frac{\al_n^{(\ell)}\overline{\beta_n^{(k)}}}{\lambda_n-\xi}  }{(z-\xi)}d\xi
\\ & = \frac{-1}{2\pi i} \int_{\gamma_\x^+} \frac{ \sumj \frac{y_j}{\lambda_j-\xi}\suml\left(\summ \overline{\beta_j^{(m)}}a_{\ell,m}(\xi)f_T^{(\ell,k)}(\xi)\right) }{(z-\xi)}d\xi,
\end{split}
\end{equation*}
for each $z \in (F_\x^+)^c$ and $k\in \N$.

Hence, equality \eqref{igualdad condicion (iii)} turns into
\begin{eqnarray} \label{igualdad condicion (iii) nueva--}
\displaystyle \sum_{n \in N_{F_\x^+}} \frac{y_n\overline{\beta_n^{(k)}}}{\lambda_n-z} & = & \displaystyle \frac{1}{2\pi i} \int_{\gamma_\x^+} \frac{ \sumj \frac{y_j}{\lambda_j-\xi}\left(\summ \overline{\beta_j^{(m)}}a_{k,m}(\xi)\right)}{z-\xi}d\xi + \nonumber \\
&& \displaystyle  + \,  \frac{1}{2\pi i} \int_{\gamma_\x^+} \frac{ \sumj \frac{y_j}{\lambda_j-\xi}\suml\left(\summ \overline{\beta_j^{(m)}}a_{\ell,m}(\xi)f_T^{(\ell,k)}(\xi)\right) }{(z-\xi)}d\xi \nonumber \\
& = & \displaystyle \frac{1}{2\pi i} \int_{\gamma_\x^+} \frac{ \sumj \frac{y_j}{\lambda_j-\xi}\suml\left(\summ \overline{\beta_j^{(m)}}a_{\ell,m}(\xi)(\delta_{\ell,k}
+ f_T^{(\ell,k)}(\xi))\right) }{(z-\xi)} d\xi \nonumber \\
& = & \displaystyle \frac{1}{2\pi i} \int_{\gamma_\x^+} \frac{\sumj \frac{y_j}{\lambda_j-\xi}\overline{\beta_j^{(k)}}}{z-\xi}d\xi,
\end{eqnarray}
where the last equality follows upon applying Theorem \ref{teorema formulas}, identity \eqref{alien cofactor} with the vector $\summ \overline{\beta_j^{(m)}}e_m \in H$ for each $j\in \N.$

So, proving identity \eqref{igualdad condicion (iii)} is equivalent to prove the identity:
\begin{equation} \label{igualdad condicion (iii) nueva}
\sum_{n \in N_{F_\x^+}} \frac{y_n\overline{\beta_n^{(k)}}}{\lambda_n-z}=  \frac{1}{2\pi i} \int_{\gamma_\x^+} \frac{\sumj \frac{y_j}{\lambda_j-\xi}\overline{\beta_j^{(k)}}}{z-\xi}d\xi,
\end{equation}
for each $z \in (F_\x^+)^c$ and $k\in \N$. At this point, we argue as in \cite[p. 24]{GG3} to check that such equality holds.


Observe that
\begin{equation}\label{ecuacion 23-03}
 \int_{\gamma_\x^+} \frac{\sumj \frac{y_j}{\lambda_j-\xi}\overline{\beta_j^{(k)}}}{z-\xi}d\xi = \sumj y_j\overline{\beta_j^{(k)}} \int_{\gamma_{\x^+}}
 \frac{d\xi}{(\lambda_j-\xi)(z-\xi)}
\end{equation}
for every $z \in (F_\x^+)^c$ since
\begin{equation*}
\begin{split}
\sumj \int_{\gamma_{\x}^+} \left| \frac{y_j\overline{\beta_j^{(k)}}}{(\lambda_j-\xi)(z-\xi)}\right||d \xi|      & =  \sumj |y_j\beta_j^{(k)}|
\int_{\gamma_{\x}^+}  \frac{|d \xi|}{|\lambda_j-\xi||z-\xi|} \\
& \leq \frac{\norm{y}}{\textnormal{dist}(z,\gamma_{\x}^+)} \left(\sumj \left( \int_{\gamma_{\x}^+}
\frac{|d \xi|}{|\lambda_j-\xi|} \right)^2 |\beta_j^{(k)}|^2\right)^{1/2},
\end{split}
\end{equation*}	
being the latter series convergent because $\x$ belongs to the decomposability set.

\medskip

Accordingly, proving \eqref{igualdad condicion (iii) nueva} turns into showing
$$ \sum_{n\in N_{F_\x^+}} \frac{y_n\overline{\beta_n^{(k)}}}{\lambda_n-z} = \sumj y_j\overline{\beta_j^{(k)}} \frac{1}{2\pi i}\int_{\gamma_{\x}^+}
\frac{d\xi}{(\lambda_j-\xi)(z-\xi)}$$ for each $z \in (F_\x^+)^c$ and $k\in \N$, which holds upon applying Cauchy's integral formula.
%

Thus, condition $(iii)$ is fulfilled.

It remains to show that condition $(iv)$ holds. That is, we have to prove that, for each $z\in (F_\x^+)^c$,
$$\sumn \left| \sumk g_k(z)\al_n^{(k)}\right|^2 <\infty.$$
So, let $z\in (F_\x^+)^c$. Having in mind  \eqref{igualdad de series}, it follows that
\begin{equation}\label{check condition (iv)}
	\sumn \left| \sumk g_k(z)\al_n^{(k)}\right|^2 = \sumn \left| \sumk   \frac{1}{2\pi i} \int_{\gamma_\x^+} \frac{ \sumj \frac{y_j}{\lambda_j-\xi}\left(\summ \overline{\beta_j^{(m)}}a_{k,m}(\xi)\right)}{z-\xi} d\xi \al^{(k)}_n \right|^2,
\end{equation}
which turns out to be equivalent to the series \eqref{acotacion J1} considered in the proof of Theorem \ref{idempotentes acotados}, where $z-\xi$ plays the role of $\lambda_n-\xi.$ Hence, arguing analogously as in the proof of Theorem \ref{idempotentes acotados} and taking into account that
$$\sumn \sumk \int_{\gamma_{\x^+}} \frac{|d\xi|}{|z-\xi|^2} |\al_n^{(k)}|^2 \leq  \textnormal{length}(\gamma_\x^+) \frac{1}{d^2} \sumn \sumk|\al_n^{(k)}|^2 <\infty,$$ with $d= \textnormal{dist}(z,\gamma_\x^+)>0,$ it follows that \eqref{check condition (iv)} holds.

\medskip

Finally, Theorem \ref{caracterizacion subespacios espectrales} yields that $x=J_\x^+ y \in H_T(F_\x^+)$, so $\ran(J_\x^+)\subset H_T(F_\x^+)$, as we wished to prove.
\end{proof}

\medskip

Now, the goal is proving that for almost every $\x$ in decomposability set  the operators $J_\x^+$ are spectral idempotents.
The proof is based on the following two lemmas:

\smallskip

\begin{lemma}\label{lema suma}
Under the hypotheses of Theorem \ref{idempotentes acotados}, for every $\x \in \Delta(T)$
$$J_\x^++J_\x^- =Id_H.$$
\end{lemma}

\smallskip

\begin{proof}
Let $\x\in \Delta(T)$ be fixed. In order to prove $J_\x^++J_\x^- =Id_H$, it is enough to show that
$$(J_\x^++J_\x^-)e_N = e_N$$
for every $N \in \N.$ Note that
\begin{equation}\label{ecuacion suma}
(J_\x^++J_\x^-)e_N = e_N + \sumn \sumk \frac{1}{2\pi i} \int_{\T} \frac{\frac{1}{\lambda_N-\xi} \summ \overline{\beta_N^{(m)}} a_{k,m}(\xi)}{\lambda_n-\xi} \al_n^{(k)}d\xi \, e_n,
 \end{equation}
 so let us show that the second summand in the right-hand side of \eqref{ecuacion suma} is zero. Arguing as in \eqref{acotacion J1} in the proof of Theorem \ref{idempotentes acotados},  we can change the series indexed in $n$ and $k$ with the integral and prove that
 \begin{equation}\label{integral 0}
\frac{1}{2\pi i}\int_{\T} \sumn \sumk \frac{\frac{1}{\lambda_N-\xi} \summ \overline{\beta_N^{(m)}} a_{k,m}(\xi)}{\lambda_n-\xi} \al_n^{(k)}\, e_n d\xi=0
 \end{equation}
for every $N\in \N$.
For this purpose, the strategy will be finding the inverse of the operator $T-z I$ for $z \in \T$ and applying an argument involving the Dunford Functional Calculus.

\medskip

For each $z \in \T$, let $B(z)$ be the operator
\begin{equation} \label{definicion B(z)}
B(z)=  \sum_{i=1}^\infty \sumj a_{i,j}(z)  (D_\Lambda -zI)^{-1}u_i\otimes (D_\Lambda^*-\overline{z}I)^{-1}v_j.
\end{equation}
Since $\overline{\Lambda} \subset \D,$ it follows that $(D_\Lambda-zI)^{-1/2}$ and $(D_\Lambda-zI)^{-1}$ are well defined bounded operators for every $z\in \T$ and moreover,   Proposition \ref{proposicion operadores X,Y} yields that the operator
$$L(z)= X(z)(I+Y(z)X(z))^{-1}Y(z)$$
is bounded. Note that
 \begin{equation}\label{definicion L}
 	L(z)= X(z)(I+Y(z)X(z))^{-1}Y(z) = \sumi \sumj a_{i,j}(z)(D_\Lambda-zI)^{-1/2}u_i\otimes (D_\Lambda^*-\overline{z}I)^{-1/2}v_j,
 \end{equation}
so $B(z) = (D_{\Lambda}-zI)^{-1/2}L(z)(D_{\Lambda}-zI)^{-1/2}$ and hence, a bounded operator.

\smallskip

Now, we claim that for each $z \in \T$
$$R(z)= (D_\Lambda-zI)^{-1} -B(z)$$
is the inverse of $T-zI$.  Note that for each $z \in \T$, the operator $R(z)$ is bounded since $\sigma(D_\Lambda)\subset \D$, so in order to prove the claim, it suffices to show that $R(z)$ is the right inverse of $(T-zI)$.

\smallskip

Let $z \in \T$ be fixed. Having in mind that $(T-zI) = (D_\Lambda-zI)^{1/2}(I+X(z)Y(z))(D_\Lambda-zI)^{1/2}$ and \eqref{definicion L}, we deduce
 \begin{equation*}
\begin{split}
(T-zI)R(z) &= (T-zI)( (D_\Lambda-zI)^{-1}-B(z))
\\ & = (T-zI)\left( (D_\Lambda-zI)^{-1}-(D_\Lambda-zI)^{-1/2}L(z)(D_\Lambda-zI)^{-1/2}\right)
\\ & = (T-zI)(D_\Lambda-zI)^{-1/2} \left(I-L(z)\right)(D-zI)^{-1/2}
\\ & = (D_\Lambda-zI)^{1/2}(I+X(z)Y(z))(D_\Lambda-zI)^{1/2}(D_\Lambda-zI)^{-1/2} \left(I-L(z)\right)(D_\Lambda-zI)^{-1/2}
\\ & = (D_\Lambda-zI)^{1/2}(I+X(z)Y(z))(I-L(z))(D_\Lambda-zI)^{-1/2}
\\ & =(D_\Lambda-zI)^{1/2}(I+X(z)Y(z))(I- X(z)(I+Y(z)X(z))^{-1}Y(z))(D_\Lambda-zI)^{-1/2}.
\end{split}
 \end{equation*}
At this point, it is enough to notice that $(I+X(z)Y(z))^{-1} = (I- X(z)(I+Y(z)X(z))^{-1}Y(z))$ (see, for instance, \cite[p. 199]{Conway}) to obtain that $(T-zI)R(z) = Id_H,$ as claimed.

\medskip

Now, by means of the Dunford functional calculus, we have that
$$\frac{-1}{2\pi i}\int_\T R(\xi) d\xi = Id_H \qquad \frac{-1}{2\pi i}\int_\T (D_\Lambda-\xi I)^{-1} d\xi = Id_H,$$ so
\begin{equation}\label{integral B(z)}
\frac{1}{2\pi i}\int_\T B(\xi) d\xi = 0.
\end{equation}

Having in mind \eqref{definicion B(z)}, we observe that
$$B(z)e_N = \sumn \sumk \frac{\frac{1}{\lambda_N-z} \summ \overline{\beta_N^{(m)}} a_{k,m}(z)}{\lambda_n-z} \al_n^{(k)}\, e_n$$
for every $z \in \T.$

Thus, \eqref{integral 0} turns into $$0 = \frac{1}{2\pi i} \int_{\T}B(\xi)e_N d\xi,$$ which holds by \eqref{integral B(z)} and concludes the proof.
\end{proof}

\begin{lemma}\label{lema producto}	Under the hypotheses of Theorem \ref{idempotentes acotados}, for almost every $\x \in \Delta(T)$
$$J_\x^+(H_T(F_\x^-)) = J_\x^-(H_T(F_\x^+)) = \{0\}.$$
Consequently, for almost every $\x \in \Delta(T)$,
$$J_\x^+J_\x^- = J_\x^-J_\x^+ = 0.$$
\end{lemma}

\medskip

 This proof is based on that  of \cite[Lemma 2.14]{GG3}, so we will highlight the differences instead of repeating some arguments.

\medskip

\begin{proof}
We are showing that $J_\x^+(H_T(F_\x^-)) = \{0\}$ and $J_\x^+J_\x^- = 0$, the other equalities are equivalent.
	
First, following the computations of the proof of \cite[Lemma 2.14]{GG3} and using \eqref{alien cofactor}, it is easy to check that $J_\x^+ \in \conm{T},$ and as a consequence $J_\x^+(H_T(F_\y^-)) = \{0\}$ for every $\y \in \Delta(T)$ and $\y<\x.$

We claim that for almost every $\x \in \Delta(T)$, there exists a sequence $(\x_n)_{n\in\N} \subset \Delta(T)$ such that $\x_n > \x$, $\x_n \searrow \x$ and $J_{\x_n}^+\rightarrow J_\x^+$ in the weak operator topology. Note that, in particular, the claim implies the equality $J_\x^+(H_T(F_\x^-)) = \{0\}$ for almost every $\x \in \Delta(T)$.

\smallskip	

In order to prove the claim, let $p \in \N$. Egorov Theorem yields the existence of a measurable subset $\X_p \subset \Delta(T)$, such that the Lebesgue measure of $\Delta(T)\setminus\X_p$ is strictly smaller than $1/p$ and the double series
	\begin{equation}\label{convergencia uniforme}
		\sumn \sumk \left(\frac{|\al_n^{(k)}|^2}{|\PR(\lambda_n)-\x|}+\frac{|\beta_n^{(k)}|^2}{|\PR(\lambda_n)-\x|}\right)
	\end{equation}
converges uniformly and is uniformly bounded on $\X_p$. Define
	$$D_p := \{z \in \overline{\D}: \PR(z)\in  \X_p\}.$$

Arguing as in Proposition \ref{continuidad} and having the uniform convergence and boundedness of \eqref{convergencia uniforme}, it follows that, indeed, the maps
$$z\in D_p \mapsto (I+Y(z)X(z))^{-1} \in \EL(H)$$
are continuous for each $p\in \N.$
	
Likewise, reasoning as in the proof of \cite[Lemma 2.14]{GG3}, it suffices to show the claim for those  $\x \in \X'_{p_0}$ where $p_0 \in \N$. In this situation, there exists a sequence $(\x_n)_{n\in\N}\subset \X_{p_0}$ such that $\x_n\nearrow \x$. Let us prove that $J_{\x_n}^+ \rightarrow J_\x^+$ in the weak operator topology.

Consider
	$$ \Gamma = \bigcup_{n \in \N} \gamma_{\x_n}^+ \cup \gamma_\x^+$$
and observe that it is a compact set. In particular, the map $z \in \Gamma \mapsto \norm{(I+Y(z)X(z))^{-1}}^2$ is continuous, so there exists $A>0$ such that the constants $C_{\x_n}^+$ given in $\eqref{acotacion condicional}$ are bounded by the constant $A$.
	
\smallskip

In order to show the weak convergence of $(J_{\x_n}^+)_{n\geq 1}$, we first check that the norm sequence $(\|J_{\x_n}^+\|)_{n\geq 1}$ is uniformly bounded. Indeed, taking into account \eqref{acotacion norma}, it is enough to recall the uniform bound of $C_{\x_n}^+$ by $A,$ and the uniform bounds of the series
$$\sumn \sumk \frac{|\al_n^{(k)}|^2}{|\PR(\lambda_n)-\x_n|} \qquad \text{ and } \qquad \sumn \sumk \frac{|\beta_n^{(k)}|^2}{|\PR(\lambda_n)-\x_n|}$$
given by the uniform convergence of \eqref{convergencia uniforme} in $\X_{p_0}.$ Finally, the constant $C>0$ given in \eqref{integral} was obtained in the proof of \cite[Theorem 2.8]{GG3}, and it  does not depend on $\x.$
Thus, $\norm{J_{\x_n}^+}$ are uniformly bounded.

\smallskip

Let us prove that $J_{\x_n}^+\rightarrow J_{\x}^+$ in the weak operator topology. Indeed, it is enough to show that
				\begin{equation}\label{convergencia debil}
					\pe{J_{\x_n}^+e_i,e_j} \rightarrow \pe{J_\x^+e_i,e_j}
				\end{equation}
				 for every $i,j\in \N.$ Now,
			\begin{equation*}
				\begin{split}
					\pe{J_{\x_n}^+e_i,e_j}
					& = \tilde{\delta}_{n,i,j}+ \sumk \frac{1}{2\pi i} \int_{\gamma_{\x_n}^+} \frac{\summ \overline{\beta_i^{(m)}}a_{k,m}(\xi)}{(\lambda_i-\xi)(\lambda_j-\xi)}d\xi \al_j^{(k)}
					\\& = \tilde{\delta}_{n,i,j}+ \frac{1}{2\pi i} \int_{\gamma_{\x_n}^+} \frac{\sumk \summ \overline{\beta_i^{(m)}}a_{k,m}(\xi)\al_j^{(k)}}{(\lambda_i-\xi)(\lambda_j-\xi)}d\xi
					\\& = \tilde{\delta}_{n,i,j}+\frac{1}{2\pi i} \int_{\gamma_{\x_n}^+} \frac{\pe{(I+Y(\xi)X(\xi))^{-1}b_i,a_j}}{(\lambda_i-\xi)(\lambda_j-\xi)}d\xi,
				\end{split}
			\end{equation*}
			where $\tilde{\delta}_{n,i,j} =1$ if $i=j \in N_{F_{\x_n}^+}$ and $0$ otherwise, $b_i = \summ \overline{\beta_i^{(m)}}e_m$ and $a_j = \sumk \overline{\al_j^{(k)}}e_k$ for every $i,j \in \N.$ Finally, observe that the map $\xi \in \Gamma \mapsto \pe{(I+Y(\xi)X(\xi))^{-1}b_i,a_j}$ is continuous, so a standard argument involving Lebesgue's Dominated Convergence Theorem shows the convergence of \eqref{convergencia debil}.

			\medskip

As a consequence, we have shown that $J_\x^+(H_T(F_\x^-)) = \{0\}$ for almost every $\x\in \Delta(T)$, which finishes the first part of the proof.
			
			\medskip
			
Finally, let $\x\in \Delta(T)$ such that $J_\x^+(H_T(F_\x^-)) = \{0\}$ and observe that $\ran(J_\x^-)\subset H_T(F_\x^-)$. It is trivial then that $J_\x^+J_\x^- = 0$, and  Lemma \ref{lema producto} is proved. 		
\end{proof}

\smallskip

We are now in position to prove the announced result of this section:

\begin{theorem}\label{teorema idempotentes}
Under the hypotheses of Theorem \ref{idempotentes acotados}, for almost every $\x\in \Delta(T)$
\begin{equation}\label{ecuacion identidades}
J_\x^++J_\x^- = Id_H, \qquad J_\x^+J_\x^- = J_\x^-J_\x^+ = 0.
\end{equation}
In particular, for almost every $\x\in \Delta(T)$ the operators $J_\x^+$ and $J_\x^-$ are  non-trivial spectral idempotents for $T$.
\end{theorem}

\smallskip

\begin{proof}
First, Lemmas \ref{lema suma} and \ref{lema producto} yields that
for almost every $\x\in \Delta(T)$ the operators $J_\x^+$ and $J_\x^-$ are idempotents satisfying \eqref{ecuacion identidades}.

Let us show that they are spectral idempotents for $T$. Let us take $\x \in \Delta(T)$ such that $J_\x^+(H_T(F_\x^-)) = \{0\}$ and $J_\x^++J_\x^- = 0$. Following the arguments of \cite[Corollary 2.15]{GG3}, it follows that $J_\x^+ \in \biconm{T}$ and $\ran(J_\x^+) = H_T(F_\x^+)$. Accordingly,  $H_T(F_\x^+)$ is a closed subspace and  $J_\x^+$ is a spectral idempotent for $T$.

Finally, let us show that $J_\x^+$ is non-trivial for every $\x\in \Delta(T)$. By the equality $J_\x^++J_\x^- = Id_H$, it is clear that if $J_\x^+$ is non-trivial, $J_\x^-$ is non-trivial as well.

Assume that $J_\x^+ =Id_H.$ If $J_\x^+=0$ then $J_\x^- = Id_H$ and the argument is analogous. Since $\ran(J_\x^+) = H_T(F_\x^+)$, it follows that $H_T(F_\x^+) = H.$ In particular, $\sigma(T) = \sigma(T\mid_{H_T(F_\x^+)}) \subseteq F_\x^+$, which is a contradiction: recall that $a< \x,$ where $a = \min\limits_{z \in \Lambda'} \PR(z)$.  So, there exists $\lambda \in \sigma(T)$ such that $\PR(\lambda)< \x$ and hence, $\lambda \notin F_\x^+.$ This ends the proof.
\end{proof}

\smallskip

\section{Proof of the Main Theorem}

Finally, we are in position to prove the Main Theorem stated in the introductory section, which we recall again for the sake of clarity:

\begin{theorem*}
Let $H$ be a separable, infinite dimensional complex Hilbert space, $\Lambda = (\lambda_n)_{n\geq 1}\subset \mathbb{C}$ a bounded sequence and $\{u_k\}_{k\geq 1}$, $\{v_k\}_{k\geq 1}$  non zero vectors in $H$. Assume
			 \begin{equation}\label{condicion---}
			 	\sum_{(n,k) \in \mathcal{N}_u} |\al_n^{(k)}|^2\log  \left( 1+ \frac{1}{|\al_n^{(k)}|}\right) + \sum_{(n,k) \in \mathcal{N}_v}|\beta_n^{(k)}|^2\log  \left(1+ \frac{1}{|\beta_n^{(k)}|}\right) < \infty,
			 \end{equation}
			 where $\mathcal{N}_u := \{(n,k) \in \N\times \N : \al_n^{(k)} \neq 0 \}$ and $\mathcal{N}_v := \{(n,k) \in \N\times \N : \beta_n^{(k)} \neq 0 \}$.
Then, the trace-class perturbation of $D_\Lambda$, $T = D_\Lambda + \sumk u_k\otimes v_k$, acting on $H$ by
\begin{equation}\label{forma---}
T x= \left (D_\Lambda + \sumk u_k\otimes v_k \right ) x= D_\Lambda x +  \sumk \pe{x,v_k} u_k, \qquad (x\in H),
\end{equation}
has non trivial closed hyperinvariant subspaces provided that it is not a scalar multiple of the identity operator. Moreover, if both point spectrum $\sigma_p(T)$ and $\sigma_p(T^*)$ are at most countable, $T$ is decomposable.
\end{theorem*}

\smallskip

\begin{proof} First of all, assume that $T$ is not a scalar multiple of the identity. For the first half, we may assume that $\sigma_p(T)\cup \sigma_p(T^*)$ is empty, since otherwise $T$ or $T^*$ will have eigenvalues and then $T$ will have non-trivial closed hyperinvariant subspaces. Moreover, one may assume that $\Lambda'$ is not a singleton: if $\Lambda' = \{ \lambda\}$, then $D_\Lambda = \lambda Id_H + \hat{K},$ where $\hat{K}$ is a compact operator. Thus, $T= \lambda I + \hat{K}+K,$ so $T$ commutes with the non-zero  compact operator $\hat{K}+K$ and Lomonosov's Theorem (see \cite{Lomonosov} or \cite[Corollary 8.24]{RR}) provides a non-trivial closed hyperinvariant subspace for $T$. In such a case, the decomposability also follows since the spectrum of $T$ is totally disconnected \cite[Proposition 1.4.5]{LN00}.

\medskip

Likewise, we also may assume that $\sigma(D_\Lambda)$ and $\sigma(T)$ are contained in $\D$ and are not contained in a vertical line, since these properties may be achieved by traslating and multiplying the operator by a scalar, which do not change the existence of non-trivial closed hyperinvariant subspaces and or the decomposability of $T$. In these conditions, we can apply Theorem \ref{teorema idempotentes} to obtain non-trivial spectral idempotents $J_\x^+$ and $J_\x^-$ for $T$ for almost every $\x\in \Delta(T)$. In particular, observe that $H_T(F_\x^+) = \ran(J_\x^+)$ and $H_T(F_\x^-)=\ran(J_\x^-)$ are non-trivial closed hyperinvariant subspaces for $T$.

\medskip

Finally, for the second half, we may assume that $\sigma_p(T)\cup \sigma_p(T^*)$ is at most countable. Then, in order to show that $T$ is decomposable, it is enough to mimic the construction of \cite[Section 3]{GG3} in order to define the non-trivial spectral idempotents for $T$ associated to horizontal lines (instead of vertical lines as $J_\x^+$) and generate a Boolean algebra of spectral idempotents associated to rectangles in $\C$. At this point, it remains to apply the arguments exposed in the proof of \cite[Theorem 3.2]{FJKP11} to obtain the decomposability of $T$, which yields the statement of the Main Theorem.
\end{proof}

\medskip

\subsection{A local version of the result}

In this subsection, we prove a local version of the Main Theorem  that will allow us to obtain non-trivial closed hyperinvariant subspaces replacing the  assumption \eqref{condicion---} by a weaker local summability condition. The result, which extends \cite[Corollary 2.8]{GG2} to trace class perturbations of diagonalizable normal operators, reads as follows:

\begin{theorem}
Let $\Lambda = (\lambda_n)_{n\geq 1}\subset \mathbb{C}$ be a bounded sequence not lying in any vertical line such that $\Lambda'$ is not a singleton  and denote  $a:= \min\limits_{z \in \Lambda'} \PR(z)$ and $b:= \max\limits_{z \in \Lambda'} \PR(z).$ Let $u_k = \sumn \al_n^{(k)}e_n$ and $v_k = \sumn \beta_n^{(k)} e_n$ be non-zero vectors in $H$ for each $k\geq 1$.  Assume that $K = \sumk u_k \otimes v_k$ is trace class and the compact perturbation  $T = D_\Lambda + K$  satisfies both $\sigma(D_\Lambda)=\overline{\Lambda}$ and $\sigma(T)$ contained in $\D$.
If there exist $a<\x_1<\x_2<b$ such that
	\begin{equation}\label{condicion local}
		\sumn \sumk \left( \frac{|\al_n^{(k)}|^2}{|\PR(\lambda_n)-\x_i|}+\frac{|\beta_n^{(k)}|^2}{|\PR(\lambda_n)-\x_i|}\right) < \infty \qquad (i=1,2).
	\end{equation}
then $T$ has non-trivial closed hyperinvariant subspaces.
\end{theorem}

\smallskip

\begin{proof}
We may assume that $\sigma_p(T)\cup\sigma_p(T^*)$ is empty. A careful reading of the proofs of Theorems \ref{idempotentes acotados} and \ref{proposicion rango} yields that the operators $J_{\x_2}^-$ and $J_{\x_2}^+$ are well-defined bounded operators such that $\ran(J_{\x_2}^-)\subset H_T(F_{\x_2}^-)$ and $\ran(J_{\x_2}^+)\subset H_T(F_{\x_2}^+)$.

Moreover, arguing as in the proof of Lemma \ref{lema suma}, it follows that $J_{\x_2}^++J_{\x_2}^- = Id_H$. Likewise,  the arguments shown in the proof of Theorem \ref{teorema idempotentes} yield that $J_{\x_2}^+$ is a non-trivial operator. We remark  that we cannot assure these operators are idempotent, since the uniform convergence of the series \eqref{convergencia uniforme} is used to prove Lemma \ref{lema producto}. Therefore, we know that $H_T(F_{\x_2}^+)$ is a non-zero spectral subspace. At this point, observe that the same construction can be carried out for the adjoint $T^*$, since it also verifies \eqref{condicion local}. As a consequence, $H_{T^*}(F_{\x_1}^-)$ is also a non-trivial spectral subspace.
	
Finally, by means of \cite[Proposition 2.5.1]{LN00}, it follows that $H_T(F_{\x_2}^+)$ is non-dense, since $H_{T^*}(F_{\x_1}^-)$ is non-zero and $F_{\x_1^-}\cap F_{\x_2^+} = \emptyset.$ Accordingly,  $\overline{H_T(F_{\x_2}^+)}$ is a non-trivial closed hyperinvariant subspace for $T$  which concludes the proof.
\end{proof}

\medskip

\section*{An open question}

In order to conclude, we note that the Main Theorem does not include every operator $T$ in the class $(\mathcal{N}+\mathcal{C}_1)$, that is, the class of trace class perturbations of diagonalizable normal operators. Therefore, we close this manuscript with the following question:

\medskip

\begin{question}
Let $H$ be a separable, infinite dimensional complex Hilbert space.  Does every operator $T$ in the class $(\mathcal{N}+\mathcal{C}_1)$  have non-trivial closed hyperinvariant subspaces? If so, if $\sigma_p(T)\cup \sigma_p(T^*)$ is at most countable, is $T$  decomposable?
\end{question}

\medskip

\section*{Acknowledgements}

The authors would like to express their gratitude to the referee for their meticulous review and constructive feedback, which significantly enhanced the quality of the manuscript.

\end{document}